\documentclass[11pt,a4paper]{article}
\usepackage[utf8]{inputenc}
\usepackage[english]{babel}
\usepackage[T1]{fontenc}
\usepackage{amsmath}
\usepackage{amsfonts}
\usepackage{amssymb}
\usepackage{amsthm}
\usepackage{mathrsfs}
\usepackage{dsfont}
\usepackage[left=2cm,right=2cm,top=3cm,bottom=3cm]{geometry}
\usepackage{accents}
\usepackage{enumitem}

\newcommand{\ubar}[1]{\underaccent{\bar}{#1}}

\RequirePackage[colorlinks,citecolor=blue,urlcolor=blue]{hyperref}

\allowdisplaybreaks

\newtheorem{theorem}{Theorem}[section]
\newtheorem{lemma}[theorem]{Lemma}
\newtheorem{proposition}[theorem]{Proposition}

\theoremstyle{remark}
\newtheorem{remark}[theorem]{Remark}

\numberwithin{equation}{section}

\DeclareMathOperator{\argmin}{argmin}
\DeclareMathOperator{\tr}{Tr}
\DeclareMathOperator{\dtv}{d_{TV}}
\newcommand{\lip}[1]{[#1]_{\textup{Lip}}}

\providecommand{\keywords}[1]
{
	\textbf{\textit{Keywords--}} #1
}

\providecommand{\MSC}[1]
{
	\textbf{\textit{MSC Classification--}} #1
}

\makeatletter
\newcommand{\owntag}[2][\relax]{% \owntag[short label]{tag}
  \ifx#1\relax\relax\def\owntag@name{#2}\else\def\owntag@name{#1}\fi% base label
  \refstepcounter{equation}\tag{\theequation, #2}%
  \expandafter\ltx@label\expandafter{eq:\owntag@name}%
  \def\@currentlabel{\theequation, #2}\expandafter\ltx@label\expandafter{Eq:\owntag@name}%
  \def\@currentlabel{#2}\expandafter\ltx@label\expandafter{tag:\owntag@name}%
}
\makeatother

\title{Convergence of Langevin-Simulated Annealing algorithms with multiplicative noise II: Total Variation}
\author{Pierre Bras\footnote{Sorbonne Universit\'e, Laboratoire de Probabilit\'es, Statistique et Mod\'elisation, UMR 8001, case 158, 4 pl. Jussieu, F-75252 Paris Cedex 5, France. E-mail: \texttt{pierre.bras@sorbonne-universite.fr} and \texttt{gilles.pages@sorbonne-universite.fr}.} \footnote{Corresponding author.} $\ $ and Gilles Pag\`es\footnotemark[1]}
\date{}

\begin{document}

\maketitle

\begin{abstract}
We study the convergence of Langevin-Simulated Annealing type algorithms with multiplicative noise, i.e. for $V : \mathbb{R}^d \to \mathbb{R}$ a potential function to minimize, we consider the stochastic differential equation $dY_t = - \sigma \sigma^\top \nabla V(Y_t) dt + a(t)\sigma(Y_t)dW_t + a(t)^2\Upsilon(Y_t)dt$, where $(W_t)$ is a Brownian motion, where $\sigma : \mathbb{R}^d \to \mathcal{M}_d(\mathbb{R})$ is an adaptive (multiplicative) noise, where $a : \mathbb{R}^+ \to \mathbb{R}^+$ is a function decreasing to $0$ and where $\Upsilon$ is a correction term. Allowing $\sigma$ to depend on the position brings faster convergence in comparison with the classical Langevin equation $dY_t = -\nabla V(Y_t)dt + \sigma dW_t$.
In a previous paper we established the convergence in $L^1$-Wasserstein distance of $Y_t$ and of its associated Euler scheme $\bar{Y}_t$ to $\argmin(V)$ with the classical schedule $a(t) = A\log^{-1/2}(t)$. 
In the present paper we prove the convergence in total variation distance. The total variation case appears more demanding to deal with and requires regularization lemmas.
\end{abstract}


\MSC{62L20, 65C30, 60H35}

\section{Introduction}

Langevin-based algorithms are used to solve optimization problems in high dimension and have gained much interest in relation with Machine Learning. The Langevin equation is a Stochastic Differential Equation (SDE) which consists in a gradient descent with noise. More precisely, let $V :  \mathbb{R}^d \rightarrow \mathbb{R}^+$ be a coercive potential function, then the associated Langevin equation reads
$$ dX_t = -\nabla V(X_t)dt + \sigma dW_t , \ t \ge 0,$$
where $(W_t)$ is a $d$-dimensional Brownian motion and where $\sigma > 0$. Under standard assumptions, the invariant measure of this SDE is the Gibbs measure $\nu_{\sigma^2}$ of density proportional to $e^{-2V(x)/\sigma^2}$ and for small enough $\sigma$, this measure concentrates around $\argmin(V)$ \cite{dalalyan2014} \cite{bras2021}.
Adding a small noise to the gradient descent allows to explore the space and to escape from traps such as local minima or saddle points appearing in non-convex optimization problems \cite{lazarev1992} \cite{dauphin2014}.
%This noise may also be interpreted as coming from the approximation of the gradient in stochastic gradient descent algorithms.
Such methods have been recently brought up to light again with Stochastic Gradient Langevin Dynamics (SGLD) algorithms \cite{welling2011} \cite{li2015}, especially for the deep learning and the calibration of large artificial neural networks.
%which is a high-dimensional non-convex optimization problem.

The Langevin-simulated annealing SDE is the Langevin equation where the noise parameter is slowly decreasing to $0$, namely
\begin{equation}
\label{eq:intro:1}
dX_t = -\nabla V(X_t) dt + a(t) \sigma dW_t , \ t \ge 0,
\end{equation}
where $a : \mathbb{R}^+ \to \mathbb{R}^+$ is non-increasing and converges to 0.
The idea is that the "instantaneous" invariant measure $\nu_{a(t)\sigma}$ which is the Gibbs measure of density $\propto \exp(-2V(x)/(a(t)^2\sigma^2))$ converges itself to $\argmin(V)$.
%This method indeed shares similarities with the original simulated annealing algorithm \cite{laarhoven1987}, which builds a Markov chain from the Gibbs measure using the Metropolis-Hastings algorithm and where the parameter $\sigma$, interpreted as a temperature, slowly decreases to zero over the iterations.
Although the additive case i.e. where $\sigma$ is constant has been extensively studied, little attention has been paid to the multiplicative case i.e. where $\sigma : \mathbb{R}^d \to \mathcal{M}_d(\mathbb{R})$ depends on $X_t$.

\medskip

The objective of the present paper is to study the convergence in total variation of the Langevin-Simulated annealing SDE, i.e. \eqref{eq:intro:1} with non-constant $\sigma$. Following \cite[Proposition 2.5]{pages2020}, we need to add a correction term in the drift, giving
\begin{align}
& dY_t = -(\sigma \sigma^\top \nabla V)(Y_t)dt + a(t) \sigma(Y_t) dW_t + \left(a^2(t) \left[\sum_{j=1}^d \partial_j(\sigma \sigma^\star)(Y_t)_{ij} \right]_{1 \le i \le d}\right) dt,\\
& a(t) = \frac{A}{\sqrt{\log(t+e)}},
\end{align}
so that $\nu_{a(t)}$ is still the the "instantaneous" invariant measure. We also study the convergence of its Euler-Maruyama scheme $\bar{Y}_t$ with decreasing steps and with noisy gradient estimates coming from stochastic gradient algorithms.
We assume in particular the convex uniformity of the potential $V$ outside a compact set (but we do not assume that the potential is convex) and the ellipticity and the boundedness of $\sigma$.

\bigskip

We studied this SDE and proved the convergence in $L^1$-Wasserstein distance of $Y$ and $\bar{Y}$ to $\nu^\star$ which is the limit measure of $\nu_a$ as $a \to 0$, in a previous paper \cite{bras2021-2}, which the present paper is a companion paper of. More precisely, we proved that $\mathcal{W}_1(Y_t,\nu^\star)$ is of order $a(t)$ as $t \to \infty$ and that $\mathcal{W}_1(Y_t,\nu_{(a(t))})$ is of order $t^{-\alpha}$ for every $\alpha \in (0,1)$.
For more details, we refer to the introduction of \cite{bras2021-2}. In particular, for applications to optimization problems arising in Stochastic Optimization and in Machine Learning and for choices of $\sigma : \mathbb{R}^d \to \mathcal{M}_d(\mathbb{R})$ used by practitioners, we refer to \cite[Section 3]{bras2021-2}.

The proof for the total variation distance case relies on the same strategy developed in \cite{bras2021-2}. We first introduce the process $X$ where the coefficient $(a(t))$ is "by plateaux" i.e. non-increasing and piecewise constant on time intervals $[T_n, T_{n+1}]$. Then we give bounds on $\dtv(X_t, Y_t)$ using a \textit{domino strategy} \cite[(1.2)]{bras2021-2}.
%For a function $f : \mathbb{R}^d \to \mathbb{R}$, the \textit{domino strategy} consists in a step-by-step decomposition of the weak error to produce an upper bound as follows:
%\begin{align}
%| \mathbb{E}f(\bar{X}_{\Gamma_n}^x) - \mathbb{E}f(X_{\Gamma_n}^x)| & = | \bar{P}_{\gamma_1} \circ \cdots \circ \bar{P}_{\gamma_n} f(x) - P_{\Gamma_n}f(x) | \nonumber \\
%% & = \left|\sum_{k=1}^n \bar{P}_{\gamma_1} \circ \cdots \circ \bar{P}_{\gamma_{k-1}} \circ (\bar{P}_{\gamma_k}-P_{\gamma_k})\circ P_{\Gamma_n-\Gamma_k}f(x) \right| \nonumber \\
%\label{eq:domino_strategy}
%& \le \sum_{k=1}^n \left|\bar{P}_{\gamma_1} \circ \cdots \circ \bar{P}_{\gamma_{k-1}} \circ (\bar{P}_{\gamma_k}-P_{\gamma_k})\circ P_{\Gamma_n-\Gamma_k}f(x) \right|,
%\end{align}
%where $P$ and $\bar{P}$ are the transition kernels associated to $X$ and $\bar{X}$ respectively and where $\Gamma_n=\gamma_1+\cdots+\gamma_n$. Then two terms appear: first the "error" term, for large $k$, where the error is controlled by classic weak and strong bounds on the error of an Euler-Maruyama scheme, and the "ergodic" term, for small $k$, where the ergodicity of $X$ is used.
However the main difference with the $L^1$-Wasserstein distance concerns the total variation distance between $X$ and $Y$ in small time as in general, it is more difficult to give bounds in small time for the total variation distance between two processes with close coefficients. Indeed, considering the functional characterization and comparing it with the $L^1$-Wasserstein distance, if $x$ and $y \in \mathbb{R}^d$ are close to each other and if $f:\mathbb{R}^d\to \mathbb{R}$ is Lipschitz-continuous, then we can bound $|f(x)-f(y)|$ by $\lip{f}|x-y|$; however if $f$ is measurable bounded, then we cannot directly bound $|f(x)-f(y)|$ in terms of $|x-y|$. Instead, the common strategy of proof in the literature is to use Malliavin calculus in order to perform an integration by parts and to use bounds on the derivatives of the density.
In this context, \cite{pages2020} relies on a highly technical Malliavin approach inducing a "regularization from the past" (see \cite[Theorem 3.7 and Appendix C]{pages2020}).

We give bounds in small time relying on the recent paper \cite{bras2021-3} and we adapt some of the proofs to the non-homogeneous Markovian setting. These bounds rely on estimates of the density of the solutions to SDE's and their derivatives \cite{friedman}. The strategy of proof is the following: we first reduce to the null drift case using a Girsanov change of measure. Then we introduce an artificial regularization in order to perform a Malliavin-type integration by parts and we use Aronson's bounds on the density and its derivatives; we need to pay attention to the dependency in the parameter $a$, controlling the ellipticity of the SDE and which converges to $0$, of the constants that appear in the Aronson bounds.
Moreover, we rely on \cite{devroye2018} to give bounds on the total variation between two Gaussian laws.

Contrary to the $L^1$-Wasserstein distance, we do not prove the convergence as $t \to \infty$ of $Y_t$ and $\bar{Y}_t$ to $\nu^\star$ since in most of the cases, $\nu^\star$ is supported by a finite number of points and then if $Y_t$ has a density then $\dtv(Y_t, \nu^\star)=2$. Instead, we prove the convergence in total variation of $Y_t$ and $\bar{Y}_t$ to their "instantaneous invariant measure" $\nu_{a(t)}$ which itself converges to $\nu^\star$ (in law, for the $L^1$-Wasserstein distance etc, see for example \cite[Theorem 2.1]{hwang1980} and \cite[Lemma 4.6]{bras2021-2}) and we give bounds on $\dtv(Y_t, \nu_{a(t)})$ and on $\dtv(\bar{Y}_t, \nu_{a(t)})$ as $t \to \infty$.

\bigskip

The paper is organized as follows.
In Section \ref{sec:assumptions_main_results} we give the setting and assumptions of the problem we consider and state our main results of convergence with convergence rates. This setting is the same as in \cite{bras2021-2}. In Section \ref{sec:small_time_bounds} we establish bounds in small time for $\dtv(X_t, Y_t)$ and for $\dtv(X_t, \bar{Y}_t)$, in inspired from \cite{bras2021-3}. In Section \ref{sec:X_convergence}, we prove the convergence of the plateaux SDE $X$ using exponential contraction properties. Using this convergence, the convergences of $\dtv(Y_t, \nu_{a(t)})$ and $\dtv(\bar{Y}_t, \nu_{a(t)})$ are proved in Section \ref{sec:Y_convergence} and \ref{sec:Y_bar_convergence} respectively.

\bigskip

\textsc{Notations}

We endow the space $\mathbb{R}^d$ with the canonical Euclidean norm denoted by $| \boldsymbol{\cdot} |$ and we denote by $\langle \cdot, \cdot \rangle$ the associated canonical inner product. For $x \in \mathbb{R}^d$ and for $R>0$, we denote $\textbf{B}(x,R) = \lbrace y \in \mathbb{R}^d : \ |y-x| \le R \rbrace$.

For $M \in (\mathbb{R}^d)^{\otimes k}$, we denote by $\|M\|$ its operator norm, i.e. $\|M\| = \sup_{u \in \mathbb{R}^{d\times k}, \ |u|=1} M \cdot u$. If $M : \mathbb{R}^d \to (\mathbb{R}^d)^{\otimes k}$, we denote $\|M\|_\infty = \sup_{x \in \mathbb{R}^d} \|M(x)\|$. We say that $M$ is $\mathcal{C}^r_b$ for some $r \in \mathbb{N}\cup\lbrace 0 \rbrace$ if $M$ is bounded and has bounded derivatives up to the order $r$.

For $k \in \mathbb{N}$ and if $f:\mathbb{R}^d \to \mathbb{R}$ is $\mathcal{C}^k$, we denote by $\nabla ^k f : \mathbb{R}^d \rightarrow (\mathbb{R}^d)^{\otimes k}$ its differential of order $k$.
If $f$ is Lipschitz-continuous, we denote by $[f]_{\text{Lip}}$ its Lipschitz constant.

%For a random vector $X$, we denote by $[X]$ its law.

We denote the total variation distance between two distributions $\pi_1$ and $\pi_2$ on $\mathbb{R}^d$:
$$ \textstyle \dtv(\pi_1,\pi_2) = 2 \sup_{A \in \mathcal{B}(\mathbb{R}^d)} |\pi_1(A) - \pi_2(A)|.$$
Without ambiguity, if $Z_1$ and $Z_2$ are two $\mathbb{R}^d$-valued random vectors, we also write $\dtv(Z_1,Z_2)$ to denote the total variation distance between the law of $Z_1$ and the law of $Z_2$.
We have as well
$$ \dtv(\pi_1,\pi_2) = \sup \left\lbrace \int_{\mathbb{R}^d} fd\pi_1 - \int_{\mathbb{R}^d} fd\pi_1, \ f : \mathbb{R}^d \to [-1,1] \ \text{measurable} \right\rbrace .$$
Moreover, we recall that if $\pi_1$ and $\pi_2$ admit densities with respect to some measure reference $\lambda$, then
$$ \dtv(\pi_1,\pi_2) = \int_{\mathbb{R}^d} \left|\frac{d\pi_1}{d\lambda} - \frac{d\pi_2}{d\lambda}\right| d\lambda .$$

We denote the $L^p$-Wasserstein distance between two distributions $\pi_1$ and $\pi_2$ on $\mathbb{R}^d$:
$$ \mathcal{W}_p(\pi_1, \pi_2) = \inf \left\lbrace \left(\int_{\mathbb{R}^d} |x-y|^p \pi(dx,dy) \right)^{1/p} : \ \pi \in \mathcal{P}(\pi_1,\pi_2) \right\rbrace ,$$
where $\mathcal{P}(\pi_1,\pi_2)$ stands for the set of probability distributions on $(\mathbb{R}^d \times \mathbb{R}^d, \mathcal{B}or(\mathbb{R}^d)^{\otimes 2})$ with respective marginal laws $\pi_1$ and $\pi_2$. For $p=1$, let us recall the Kantorovich-Rubinstein representation of the Wasserstein distance of order 1 \cite[Equation (6.3)]{villani2009}:
$$ \mathcal{W}_1(\pi_1,\pi_2) = \sup \left\lbrace \int_{\mathbb{R}^d} f(x) (\pi_1-\pi_2)(dx) : \ f : \mathbb{R}^d \to \mathbb{R}, \ [f]_{\text{Lip}} = 1 \right\rbrace .$$

For $x \in \mathbb{R}^d$, we denote by $\delta_x$ the Dirac mass at $x$.

In this paper, we use the notation $C$ and $c$ to denote real positive constants, which may change from line to line.

\section{Assumptions and main results}
\label{sec:assumptions_main_results}

\subsection{Assumptions}
\label{subsec:assumptions}

Let us briefly recall the setting adopted in \cite{bras2021-2}.
Let $V : \mathbb{R}^d \rightarrow (0,+\infty)$ be a $\mathcal{C}^2$ potential function such that $V$ is coercive and
\begin{equation}
\label{eq:def:A}
(x \mapsto |x|^2 e^{-2V(x)/A^2}) \in L^1(\mathbb{R}^d) \text{ for some } A>0.
\end{equation}
Then $V$ admits a minimum on $\mathbb{R}^d$. Moreover, let us assume that
\begin{align}
V^\star :=\min_{\mathbb{R}^d} V >0, \quad \argmin(V) = \lbrace x_1^\star, \ldots, x_{m^\star}^\star \rbrace, \quad \forall \ i =1,\ldots,m^\star, \ \nabla^2 V(x_i^\star) >0,
\owntag[eq:min_V]{$\mathcal{H}_{V1}$}
\end{align}
i.e. $\min_{\mathbb{R}^d} V$ is attained at a finite number $m^\star$ of points and at each point the Hessian matrix is positive definite. We then define for $a \in (0,A]$ the Gibbs measure $\nu_{a}$ of density :
\begin{equation}
\label{eq:def_nu}
\nu_a(dx) = \mathcal{Z}_{a} e^{-2(V(x)-V^\star)/a^2} dx , \quad \mathcal{Z}_{a} = \left( \int_{\mathbb{R}^d} e^{-2(V(x)-V^\star)/a^2} dx \right)^{-1}
\end{equation}
Following \cite[Theorem 2.1]{hwang1980}, the measure $\nu_a$ converges weakly to $\nu^\star$ as $a \to 0$, where $\nu^\star$ is the weighted sum of Dirac measures:
\begin{equation}
\nu^\star = \left(\sum_{j=1}^{m^\star} \left(\det \nabla^2 V(x_j^\star) \right)^{-1/2} \right)^{-1} \sum_{i=1}^{m^\star} \left(\det \nabla^2 V(x_i^\star)\right)^{-1/2} \delta_{x_i^\star}.
\end{equation}
Following \cite[Lemma 4.6]{bras2021-2}, $\nu_a$ also converges to $\nu^\star$ as $a \to 0$ for the $L^1$-Wasserstein distance.

\medskip

We consider the following Langevin SDE in $\mathbb{R}^d$:
\begin{align}
\label{eq:def_Y}
Y_0^{x_0} = x_0 \in \mathbb{R}^d, \quad dY_t^{x_0} =  b_{a(t)}(Y_t^{x_0})dt + a(t) \sigma(Y_t^{x_0}) dW_t,
\end{align}
where, for $a\ge 0$, the drift $b_a$ is given by
\begin{equation}
\label{eq:def_b}
b_a(x) = -(\sigma \sigma^\top \nabla V)(x) + a^2 \left[\sum_{j=1}^d \partial_j(\sigma \sigma^\top)_{ij}(x) \right]_{1 \le i \le d} =: -(\sigma \sigma^\top \nabla V)(x) + a^2 \Upsilon(x),
\end{equation}
where $W$ is a standard $\mathbb{R}^d$-valued Brownian motion defined on a probability space $(\Omega, \mathcal{A}, \mathbb{P})$, where $ \sigma : \mathbb{R}^d \to \mathcal{M}_{d}(\mathbb{R})$ is $\mathcal{C}^2$ and
\begin{equation}
\label{eq:def_a}
a(t) = \frac{A}{\sqrt{\log(t+e)}} 
\end{equation}
where $A$ is defined in \eqref{eq:def:A} and with $\log(e)=1$.
This equation corresponds to a gradient descent on the potential $V$ with preconditioning $\sigma$ and multiplicative noise ; the second term in the drift \eqref{eq:def_b} is a correction term (see \cite[Proposition 2.5]{pages2020}) which is zero for constant $\sigma$.
%For convenience let us define:
%\begin{align}
%[\Upsilon(x)]_{1 \le i \le d} := \left[\sum_{j=1}^d \partial_i(\sigma \sigma^\top)_{ij}(x)\right]_{1 \le i \le d}.
%\label{eq:def_upsilon}
%\end{align}

\medskip

We make the following assumptions on the potential $V$:
\begin{align}
|\nabla V|^2 \le CV \ \text{ and } \sup_{x \in \mathbb{R}^d} || \nabla^2 V(x)|| < + \infty ,
\owntag[eq:V_assumptions]{$\mathcal{H}_{V2}$}
\end{align}
%\lim_{|x| \rightarrow + \infty} V(x) = + \infty, \ \
which implies in particular that $V$ has at most a quadratic growth.
Let us also assume that
\begin{align}
\sigma \text{ is bounded and Lipschitz-continuous,} \ \nabla^2 \sigma \text{ is bounded}, \ \nabla(\sigma\sigma^\top) \nabla V \text{ is bounded},
\owntag[eq:sigma_assumptions]{$\mathcal{H}_\sigma$}
\end{align}
and that $\sigma$ is uniformly elliptic, i.e.
\begin{equation}
\label{eq:ellipticity}
\exists \ubar{\sigma}_0 > 0, \ \forall x \in \mathbb{R}^d, \ (\sigma \sigma^\top) (x) \ge \ubar{\sigma}_0^2 I_d .
\end{equation}
Assumptions \eqref{Eq:eq:V_assumptions} and \eqref{Eq:eq:sigma_assumptions} imply that $\Upsilon$ is also bounded and Lipschitz-continuous and that $b_a$ is Lipschitz-continuous uniformly in $a \in [0,A]$. Let the minimal constant $[b]_{\text{Lip}}$ be such that:
\begin{equation}
\forall a \in [0,A], \ b_a \text{ is } [b]_{\text{Lip}} \text{-Lipschitz continuous} .
\end{equation}

We make the non-uniform dissipative (or convexity) assumption outside of a compact set: there exists $\alpha_0 >0$ and $R_0 >0$ such that
%\begin{equation*}
%\forall x, y \in \textbf{B}(0,R_0)^c, \ \langle b_u(x)-b_u(y), \ x-y \rangle \le - \alpha |x-y|^2.
%\end{equation*}
\begin{align}
\forall x, y \in \textbf{B}(0,R_0)^c, \ \left\langle \left( \sigma \sigma^\top \nabla V\right)(x)-\left(\sigma \sigma^\top \nabla V \right) (y), \ x-y \right\rangle \ge \alpha_0 |x-y|^2.
\owntag[eq:V_confluence]{$\mathcal{H}_{cf}$}
\end{align}

%\begin{proposition}
%These assumptions imply:
%\begin{enumerate}
%	\item $\forall x \in \mathbb{R}^d$, $V(x) \le C|x-x^\star|^2 + V(x^\star)$.
%	\item $|\nabla V|$ is coercive.
%	\item $\forall x \in \mathbb{R}^d$, $V(x)-V(x^\star) \ge C|x-x^\star|^2$.
%\end{enumerate}
%\end{proposition}
%\begin{proof}
%\begin{itemize}
%	\item We have $|\nabla V|^2 \le CV$ so $\nabla(\sqrt{V}) \le C$.
%	\item Using \eqref{Eq:eq:V_confluence}, for $x$, $y \in \textbf{B}(0,R_0)^c$:
%$$ ||\sigma||_\infty^2 |\nabla V(x)| |x-y| \ge |\langle (\sigma \sigma^\star \nabla V)(x), x-y \rangle| \ge \alpha|x-y|^2 - |\langle (\sigma\sigma^\star \nabla V)(y),x-y \rangle | $$
%so
%$$ |\nabla V(x)| \ge ||\sigma||_\infty^{-2}(\alpha|x-y| - |(\sigma\sigma^\star \nabla V)(y)|) ,$$
%and taking $y$ fixed and $|x| \rightarrow \infty$ yields that $|\nabla V(x)| \rightarrow \infty$.
%	\item Let us define $G := \sigma \sigma^\star V$, then 
%\end{itemize}
%\end{proof}
Taking $y \in \textbf{B}(0,R_0)^c$ fixed, letting $|x| \to \infty$ and using the boundedness of $\sigma$, \eqref{Eq:eq:V_confluence} implies that $|\nabla V|$ is coercive.
%This implies that $|\nabla V|$ is coercive because
%for $x$, $y \in \textbf{B}(0,R_0)^c$:
%\begin{align*}
%& ||\sigma||_\infty^2 |\nabla V(x)| |x-y| \ge |\langle (\sigma \sigma^\star \nabla V)(x), x-y \rangle| \ge \alpha|x-y|^2 - |\langle (\sigma\sigma^\star \nabla V)(y),x-y \rangle | \\
%& |\nabla V(x)| \ge ||\sigma||_\infty^{-2}(\alpha|x-y| - |(\sigma\sigma^\star \nabla V)(y)|) .
%\end{align*}
Using \eqref{Eq:eq:V_assumptions} and the boundedness of $\sigma$, there exists $C>0$ (depending on $A$) such that:
$$ \forall a \in [0,A], \ 1 + |b_a(x)| \le CV^{1/2}(x) .$$

\medskip

Let $(\gamma_n)_{n \ge 1}$ be a non-increasing sequence of varying positive steps. We define $\Gamma_n := \gamma_1 + \cdots + \gamma_n$ and for $t \ge 0$:
\begin{equation}
N(t) := \min \lbrace k \ge 0 : \ \Gamma_{k+1} > t \rbrace = \max \lbrace k \ge 0 : \ \Gamma_k \le t \rbrace .
\end{equation}
We make the classical assumptions on the step sequence, namely
\begin{equation}
\gamma_n \downarrow 0, \quad \sum_{n \ge 1} \gamma_n = + \infty \quad \text{and} \quad \sum_{n \ge 1} \gamma_n^2 < + \infty
\owntag[eq:gamma_assumptions]{$\mathcal{H}_{\gamma1}$}
\end{equation}
and we also assume that
\begin{align}
\varpi := \limsup_{n \to \infty} \frac{\gamma_n - \gamma_{n+1}}{\gamma_{n+1}^2} < \infty .
\owntag[eq:gamma_assumptions_2]{$\mathcal{H}_{\gamma2}$}
\end{align}
For example, if $\gamma_n = \gamma_1/n^\eta$ with $\eta \in (1/2,1)$ then $\varpi = 0$; if $\gamma_n = \gamma_1/n$ then $\varpi = \gamma_1$.

\medskip

In stochastic gradient algorithms, the true gradient is measured with a zero-mean noise $\zeta$, which law only depends on the current position. That is, let us consider a family of random fields $(\zeta_n(x))_{x \in \mathbb{R}^d, n \in \mathbb{N}}$ such that for every $n \in \mathbb{N}$, $(\omega, x) \in \Omega \times \mathbb{R}^d \mapsto \zeta_n(x, \omega)$ is measurable and for all $x \in \mathbb{R}^d$, the law of $\zeta_n(x)$ only depends on $x$ and $(\zeta_n(x))_{n \in \mathbb{N}}$ is an i.i.d. sequence independent of $W$. We make the following assumptions:
\begin{equation}
\label{eq:zeta_assumptions}
\forall x \in \mathbb{R}^d, \ \forall p \ge 1, \ \mathbb{E}[\zeta_{1}(x)] = 0 \quad \text{and} \quad \mathbb{E}[|\zeta_{1}(x)|^p] \le C_p V^{p/2}(x).
\end{equation}
We then consider the Euler-Maruyama scheme with decreasing steps associated to $(Y_t)$:
\begin{align}
\label{eq:def_Y_bar}
& \bar{Y}_0^{x_0} = x_0, \quad \bar{Y}_{\Gamma_{n+1}}^{x_0} = \bar{Y}_{\Gamma_n} + \gamma_{n+1} \left(b_{a(\Gamma_{n})}(\bar{Y}^{x_0}_{\Gamma_{n}}) + \zeta_{n+1}(\bar{Y}_{\Gamma_n}^{x_0}) \right) + a(\Gamma_{n}) \sigma(\bar{Y}_{\Gamma_n}^{x_0})(W_{\Gamma_{n+1}} - W_{\Gamma_n}),
\end{align}
We extend $\bar{Y}^{x_0}_{\boldsymbol{\cdot}}$ on $\mathbb{R}^+$ by considering its genuine continuous interpolation:
\begin{equation}
\label{eq:def_Y_bar_genuine}
\forall t \in [\Gamma_n, \Gamma_{n+1}), \  \bar{Y}^{x_0}_{t} = \bar{Y}^{x_0}_{\Gamma_n} + (t-\Gamma_n) \left(b_{a(\Gamma_n)}(\bar{Y}^{x_0}_{\Gamma_n}) + \zeta_{n+1}(\bar{Y}_{\Gamma_n}^{x_0}) \right) + a(\Gamma_n) \sigma(\bar{Y}^{x_0}_{\Gamma_n}) (W_t - W_{\Gamma_n}) .
\end{equation}

\subsection{Main results}

%We now state our main results.

\begin{theorem}
\label{thm:main}
\begin{enumerate}[label=(\alph*)]
\item Let $Y$ be defined in \eqref{eq:def_Y}. Assume \eqref{Eq:eq:min_V}, \eqref{Eq:eq:V_assumptions}, \eqref{Eq:eq:sigma_assumptions}, \eqref{eq:ellipticity} and \eqref{Eq:eq:V_confluence}. Then, for every $\alpha \in (0,1)$, if $A$ is large enough, then for every $x_0 \in \mathbb{R}^d$ and for every $t >0$:
\begin{equation}
\dtv\left(Y_t^{x_0},\nu_{a(t)}\right) \le Ce^{C\sqrt{\log(t)}(1+|x_0|^2)} t^{-\alpha} .
\end{equation}

\item Let $\bar{Y}$ be defined in \eqref{eq:def_Y_bar}. Assume \eqref{Eq:eq:min_V}, \eqref{Eq:eq:V_assumptions}, \eqref{Eq:eq:sigma_assumptions}, \eqref{eq:ellipticity} and \eqref{Eq:eq:V_confluence}. Assume furthermore that $\sigma \in \mathcal{C}^{2r}_b$. Assume furthermore \eqref{Eq:eq:gamma_assumptions} and \eqref{Eq:eq:gamma_assumptions_2}, that $V$ is $\mathcal{C}^3$ with $\|\nabla^3 V\| \le CV^{1/2}$ and that $\sigma$ is $\mathcal{C}^3$ with $\|\nabla^3(\sigma \sigma^\top) \| \le CV^{1/2}$. Then, for every $\alpha \in (0,1)$, if $A$ is large enough, then for every $x_0 \in \mathbb{R}^d$ and for every $t >0$:
\begin{equation}
\label{eq:thm_Y_bar}
\dtv\left(\bar{Y}^{x_0}_t, \nu_{a(t)} \right) \le C \left(\log^{1/2}(t)\max\left[V^2(x_0),1+|x_0|\right]t^{-\alpha} + e^{C\sqrt{\log(t)} (1+|x_0|^2)} t^{C/A^2} \gamma_{N(Ct)}^{r/(2r+1)} \right).
\end{equation}
\end{enumerate}
\end{theorem}

\begin{remark}
Depending on the step sequence $(\gamma_n)$, we can compare the two terms arising in the right-hand side of \eqref{eq:thm_Y_bar}. For example, if $\gamma_n = \gamma_1 n^{-\eta}$ for some $\eta \in (1/2,1]$, then
\begin{itemize}
	\item If $\eta = 1$, then $\gamma_{N(Ct)} \asymp e^{-Ct}$ and the first term is the dominating term.
	\item If $\eta \in (1/2,1)$ then $\gamma_{N(Ct)} \asymp (Ct)^{-\eta/(1-\eta)}$.
\end{itemize}

\end{remark}

\subsection{Extensions and interpolations of the processes}

Let us define the following processes that will be used as auxiliary tools in the proofs.

$\bullet$ We define $(X_t)$ as the solution the following SDE where the coefficients piecewisely depend on the time; $X$ is then said to be "by plateaux":
\begin{align}
\label{eq:def_X}
& X_0^{x_0} = x_0, \quad dX_t^{x_0} = b_{a_{k+1}}(X_t^{x_0})dt + a_{k+1} \sigma(X_t^{x_0})dW_t, \quad t \in [T_k,T_{k+1}],
\end{align}
where $b_a$ is defined in \eqref{eq:def_b} and the time schedule $(T_n)$ is defined by
\begin{equation}
\label{eq:def_T_n}
T_n := C_{(T)}n^{1+\beta},
\end{equation}
where $C_{(T)}>0$, $\beta>0$ and $a_n := a(T_n)$. More generally, we define $(X^{x,n}_t)$ as the solution of
\begin{equation}
\label{eq:def_X:2}
X^{x,n}_0 = x, \quad dX_t^{x,n} = b_{a_{k+1}}(X_t^{x,n})dt + a_{k+1} \sigma(X_t^{x,n})dW_t, \quad t \in [T_k-T_n,T_{k+1}-T_n], \ k \ge n,
\end{equation}
i.e. $(X_t^{x,n})$ has the conditional law of $(X_{T_n+t})_{t \ge 0}$ given $X_{T_n}=x$. We have $X_{t}^x = X_{t}^{x,0}$. The Markov transition kernel associated to $X^{\cdot,n}$ denoted $P^{X,n}_t$ reads on Borel functions $f : \mathbb{R}^d \to \mathbb{R}^+$, $P^{X,n}_t f(x) = \mathbb{E}[f(X^{x,n}_t)]$.

\medskip

$\bullet$ Considering now the original SDE \eqref{eq:def_Y}, we also define for every $x \in \mathbb{R}^d$ and every fixed $u \ge 0$:
\begin{align}
\label{eq:def_Y:2}
Y_{0,u}^{x} & = x, \quad dY_{t,u}^{x} = b_{a(t+u)}(Y_{t,u}^{x})dt + a(t+u) \sigma(Y_{t,u}^{x}) dW_t,
\end{align}
so that $Y^x = Y^x_{\cdot, 0}$.
We define the Markov transition kernel associated to $Y$ between the times $t$ and $t+u$ by $P^Y_{t,u}$ such that for all Borel functions $f : \mathbb{R}^d \to \mathbb{R}^+$, $P^Y_{t,u} f(x) = \mathbb{E}[f(Y^x_{t,u})]$.

\medskip

$\bullet$ Considering finally \eqref{eq:def_Y_bar} and \eqref{eq:def_Y_bar_genuine}, we define for every $n \ge 0$, $(\bar{Y}^x_{t,\Gamma_n})_{t\ge 0}$, first at times $\Gamma_k-\Gamma_n$, $k \ge n$, by
\begin{align}
\bar{Y}^x_{0,\Gamma_n} = x, \quad \bar{Y}^x_{\Gamma_{k+1}-\Gamma_n,\Gamma_n} & = \bar{Y}^x_{\Gamma_k -\Gamma_n, \Gamma_n} + \gamma_{k+1} \left(b_{a(\Gamma_k)}(\bar{Y}_{\Gamma_k-\Gamma_n,\Gamma_n}^{x}) + \zeta_{k+1}(\bar{Y}_{\Gamma_k-\Gamma_n,\Gamma_n}^{x}) \right) \nonumber \\
\label{eq:def_Y_bar:2}
& \quad + a(\Gamma_k)\sigma(\bar{Y}_{\Gamma_k-\Gamma_n,\Gamma_n}^{x})(W_{\Gamma_{k+1}} - W_{\Gamma_k}),
\end{align}
then at every time $t$ by the genuine interpolation on the intervals $([\Gamma_k-\Gamma_n, \Gamma_{k+1}- \Gamma_n))_{k \ge n}$ as before. In particular $\bar{Y}^x = \bar{Y}^x_{\cdot,0}$.
Still more generally, we define $\bar{Y}^x_{t,u}$ where $u \in (\Gamma_n, \Gamma_{n+1})$ as
\begin{equation*}
\bar{Y}^x_{0,u} = x, \quad \bar{Y}^x_{t,u} = \left\lbrace \begin{array}{ll}
 x + t(b_a(x) + \zeta_{n+1}(x)) + a^2(u)\sigma(x)(W_t-W_{\Gamma_u}) & \text{ if } t \in [u, \Gamma_{n+1}] \\
 = \bar{Y}^{\bar{Y}^x_{\Gamma_{n+1}-u,u}}_{t-(\Gamma_{n+1}-u),\Gamma_{n+1}} & \text{ if } t > \Gamma_{n+1} .
\end{array} \right.
\end{equation*}
For $n$, $k \ge 0$, for $u \in [\Gamma_k,\Gamma_{k+1})$ and $\gamma \in [0,\Gamma_{k+1}-u]$, let $P^{\bar{Y}}_{\gamma,u}$ be the Markov transition kernel associated to $\bar{Y}_{\cdot,u}$ between the times $0$ and $\gamma$ i.e. for all Borel functions $f : \mathbb{R}^d \to \mathbb{R}^+$, $P^{\bar{Y}}_{\gamma,u}f(x) = \mathbb{E}[f(\bar{Y}^x_{\gamma,u})]$.

\section{Bounds in total variation for small $t$}
\label{sec:small_time_bounds}

In this section we give bounds for the total variation distance between the processes $X$, $Y$ and $\bar{Y}$. Although such bounds are straightforward for $L^p$-distances, they are more difficult to establish for $\dtv$. To this end we adopt a strategy similar to \cite{bras2021-3}.

For $x \in \mathbb{R}^d$ and for $a \in \mathbb{R}^+$ we define the "cut" drift $\tilde{b}^x_{a} : \mathbb{R}^d \to \mathbb{R}^d$ which is the drift $b_a$ which is null outside a compact set centred on $x$. More precisely, we choose $R > 0$ and we consider a $\mathcal{C}^\infty$ decreasing function $\psi : \mathbb{R}^+ \to \mathbb{R}^+$ such that $\psi = 1$ on $[0,R^2]$ and $\psi = 0$ on $[(R+1)^2,\infty)$ and we define $\tilde{b}^x_a(y) := b_a(y)\psi(|y-x|^2)$, so that $|\tilde{b}^x_a|$ is bounded by $C(1+|x|)$ since $b_a$ is Lipschitz-continuous.

For $\sigma : \mathbb{R}^d \to \mathcal{M}_d(\mathbb{R})$, we denote the martingale:
\begin{equation}
M(\sigma)^x_0 = x, \quad dM(\sigma)_t^x = \sigma(M(\sigma)^x_t) dW_t
\end{equation}
with its associated one-step Euler-Maruyama scheme:
\begin{equation}
\bar{M}(\sigma)_t^x = x + \sigma(x) W_t .
\end{equation}

\begin{lemma}
\label{lemma:change_time_martingale}
Let $Z$ be solution of the following SDE:
$$ dZ^x_t = u(t) \sigma_{_Z}(Z^x_t) dW_t ,$$
where $u:\mathbb{R}^+ \to (0,\infty)$ is $\mathcal{C}^1$ and bounded. Then $(Z_t) \sim (M(\sigma)_{F^{(-1)}(t)})$, where $F:\mathbb{R}^+ \to \mathbb{R}^+$ is solution of the differential equation
$$ F(0)=0, \quad F'(t) = \frac{1}{u^2(F(t))} $$
and where $F^{(-1)}$ denotes the (continuous) inverse function of $F$.
\end{lemma}
\begin{proof}
First, $F$ is well defined and is strictly increasing with $F(t) \to \infty$ as $t \to \infty$ since $u$ is bounded, so that $F^{(-1)} : \mathbb{R}^+ \to \mathbb{R}^+$ is well defined as well.
We have
\begin{align*}
d\left(Z^x_{F(t)}\right) & = u(F(t)) \sigma_{_Z}(Z^x_{F(t)}) d\left(W_{F(t)}\right) = F'(t)^{1/2} u(F(t)) \sigma_{_Z}(Z^x_{F(t)}) d\widetilde{W}_t = \sigma_{_Z}(Z^x_{F(t)}) d\widetilde{W}_t.
\end{align*}
where $\widetilde{W}$ is the Brownian motion defined by $\widetilde{W}_t = \int_0^t (F'(s))^{-1/2} dW_{F(s)}$.
\end{proof}

%Let us recall a result giving bounds on the total variation between processes, knowing bounds on the $L^1$-Wasserstein distance and Aronson's bounds on derivatives of the densities of both processes up to some even order.

\subsection{Total variation bound in small time for the Euler-Maruyama scheme}

\begin{proposition}
\label{prop:dtv_X_Y_bar_st}
Assume that $\sigma \in \mathcal{C}^{2r}_b$.
There exists $C>0$ such that for every $n$, $k \ge 0$, for every $u \in [\Gamma_k,\Gamma_{k+1})$ and every $t>0$ such that $u \in [T_n,T_{n+1}]$, $t \le \Gamma_{k+1}-u$ and $u+t \in [T_{n},T_{n+1}]$,
\begin{equation}
\label{eq:dtv_X_Y_bar_st}
\dtv(X^{x,n}_t, \bar{Y}^x_{t,u}) \le Ce^{Ca_{n+1}^{-1}(1+|x|^2)}t^{r/(2r+1)} + C a_n^{-2} (a_n-a_{n+1}) .
\end{equation}
\end{proposition}
\begin{proof}
We apply a strategy of proof similar to \cite[Theorem 2.2]{bras2021-3}. However we need to pay attention to the dependency of the constants in the bounds in $(a_n)$. Let us write
\begin{align}
\dtv(X^{x,n}_t, \bar{Y}^x_{t,u}) & \le \dtv(X^{x,n}_t, \widetilde{X}^{x,n}_t) + \dtv(\widetilde{X}^{x,n}_t, Z^{x,n}_t) + \dtv(Z^{x,n}_t, \bar{Z}^{x,n}_t) \nonumber \\
\label{eq:dTV_X_Y_bar}
& \quad + \dtv(\bar{Z}^{x,n}_t, \bar{X}^{x,n}_t) + \dtv(\bar{X}^{x,n}_t, \bar{Y}^x_{t,u}),
\end{align}
where
\begin{align*}
& \widetilde{X}^{x,n}_0 = x, \quad d\widetilde{X}^{x,n}_t = \tilde{b}_{a_{n+1}}^x(\widetilde{X}^{x,n}_t)dt + a_{n+1} \sigma(\widetilde{X}^{x,n}_t)dW_t, \\
& Z^{x,n}_0 = x, \quad dZ^{x,n}_t = a_{n+1} \sigma(Z^{x,n}_t) dW_t, \\
& \bar{Z}^{x,n}_0=x, \quad \bar{Z}^{x,n}_t = x + a_{n+1} \sigma(x) W_t.
\end{align*}

\medskip

$\bullet$ Using \cite[Lemma 3.2]{bras2021-3}, we have
$$ \dtv(X^{x,n}_t, \widetilde{X}^{x,n}_t) \le C(1+|x|^2)t ,$$
where the constant $C$ does not depend on $n$.

\medskip

$\bullet$ We use \cite[Theorem 2.4]{qian2003} and we rework the bound from \cite[Lemma 3.5]{bras2021-3} to make explicit the dependency in $a_n$. Reworking \cite[Lemma 3.4]{bras2021-3}, we have for $q \ge 1$:
\begin{align}
\label{eq:U_s_bound}
& \textstyle \mathbb{E}\left[\sup_{s \in [0,t]} |U_s^{x,n}|^{2q} \right] \le C e^{C_q a_{n+1}^{-1}(1+|x|^2)}, \\
& U^{x,n}_0 = 1, \quad dU^{x,n}_s = a_{n+1}^{-1} U^{x,n}_s \left\langle \sigma^{-1}(Z^{x,n}_s) \tilde{b}^x_{a_{n+1}}(Z^{x,n}_s), dW_s \right\rangle. \nonumber
\end{align}
Moreover, following Lemma \ref{lemma:change_time_martingale} we have $(Z_t^{x,n}) \sim (M(\sigma)_{F^{(-1)}(t)}^x)$ where the process $(M(\sigma)_t)$ does not depend on $n$ and where $F^{(-1)}(t) = a_{n+1}^2 t$. Thus following \cite[Chapter 9, Theorem 7]{friedman} (also see \cite[Theorem 3.1]{bras2021-3} for the application to SDE's) and since $\sigma \in \mathcal{C}^2_b$ we have
\begin{equation}
\label{eq:M_sigma_density_bound}
|\nabla_x p_{_{M(\sigma)}}(t,x,y)| \le \frac{C}{t^{(d+1)/2}} e^{-c|y-x|^2/t} 
\end{equation}
and then
\begin{align*}
|\nabla_x p_{_Z}(t,x,y)| = |\nabla_x p_{_{M(\sigma)}}(a_{n+1}^{2}t,x,y)| \le \frac{Ca_{n+1}^{-(d+1)}}{t^{(d+1)/2}} e^{-c a_{n+1}^{-2}|y-x|^2/t} \le \frac{Ca_{n+1}^{-(d+1)}}{t^{(d+1)/2}} e^{-c|y-x|^2/t}.
\end{align*}
Then using \cite[Lemma 3.5]{bras2021-3} with the adapted bound on $U_s^{x,n}$ \eqref{eq:U_s_bound} along with \cite[Theorem 2.4]{qian2003}, we obtain
\begin{equation}
\label{eq:dtv_X_tilde_Z}
\dtv(\widetilde{X}^{x,n}_t, Z^{x,n}_t) \le Ce^{C a_{n+1}^{-1}(1+|x|^2)}t^{1/2} .
\end{equation}
The same way, we obtain
$$ \dtv(\bar{Z}^{x,n}_t, \bar{X}^{x,n}_t) \le Ce^{C a_{n+1}^{-1}(1+|x|^2)}t^{1/2}. $$

\medskip

$\bullet$ Following Lemma \ref{lemma:change_time_martingale}, we have $(Z_t^{x,n}) \sim (M(\sigma)^x_{a_{n+1}^{2}t})$ and $(\bar{Z}^{x,n}_t) \sim (\bar{M}(\sigma)^x_{a_{n+1}^{2}t})$, where both processes $M(\sigma)$ and $\bar{M}(\sigma)$ do not depend on $n$.
We then use \cite[Theorem 2.2]{bras2021-3} to get
$$ \dtv(M(\sigma)^x_t,\bar{M}(\sigma)^x_t) \le Ce^{C|x|^2}t^{r/(2r+1)} $$
which implies
$$ \dtv(Z^{x,n}_t, \bar{Z}^{x,n}_t) \le Ce^{C|x|^2} a_{n+1}^{2r/(2r+1)} t^{r/(2r+1)} \le Ce^{C|x|^2}t^{r/(2r+1)}. $$

\medskip

$\bullet$ Let us now investigate $\dtv(\bar{X}^{x,n}_t, \bar{Y}^x_{t,u})$. Conditionally to $\zeta(x)$, both random vectors are Gaussian vectors with
$$ \bar{X}^{x,n}_t \sim \mathcal{N}\left(x + tb_{a_{n+1}}(x), a_{n+1}^2 t \sigma \sigma^\top(x)\right) \quad \text{and} \quad \bar{Y}^x_{t,u} \sim \mathcal{N}\left(x + tb_{a(u)}(x) + t\zeta(x), a^2(u)t \sigma \sigma^\top(x)\right) .$$
Then, conditionally to $\zeta(x)$ we have
\begin{align*}
& \dtv(\bar{X}^{x,n}_t, \bar{Y}^x_{t,u}) \le \dtv\left(\mathcal{N}\left(x + tb_{a_{n+1}}(x), a_{n+1}^2 t \sigma \sigma^\top(x)\right), \mathcal{N}\left(x + tb_{a(u)}(x) + t\zeta(x), a_{n+1}^2 t \sigma \sigma^\top(x)\right) \right) \\
& \quad + \dtv\left(\mathcal{N}\left(x + tb_{a(u)}(x) + t\zeta(x), a_{n+1}^2 t \sigma \sigma^\top(x)\right), \mathcal{N}\left(x + tb_{a(u)}(x) + t\zeta(x), a^2(u)t \sigma \sigma^\top(x)\right) \right) \\
& =: D_1 + D_2.
\end{align*}
We then refer to \cite{devroye2018} which gives bounds on the total variation between two Gaussian laws, first in the case $d>1$.
Using \cite[Theorem 1.1]{devroye2018} with $\lambda_1 = \cdots = \lambda_d = (a(u)^2-a_{n+1}^2)/a_{n+1}^2$, we have
\begin{align*}
D_2 \le C\left(\frac{a^2(u)-a_{n+1}^2}{a_{n+1}^2}\right) \le C a_n^{-1}(a_n-a_{n+1}).
\end{align*}
Using \cite[Theorem 1.2]{devroye2018}, since the $\rho_i$'s are bounded independently of $n$ and since for every $y \in \mathbb{R}^d$, $y^\top \sigma \sigma^\top(x) y \ge \ubar{\sigma}_0^2 |y|^2$, we have
\begin{align*}
D_1 & \le C\sqrt{t} a_{n+1}^{-1} (1 + |\zeta(x)|^{1/2}).
% \le C\max\left( \frac{t^2(a^2(u)-a_{n+1}^2)^2(a^2(u)-a_{n+1}^2) L(x)}{t^2(a^2(u)-a_{n+1}^2)^2 a_{n+1}^2 L(x)}, \frac{t^2(a^2(u)-a_{n+1}^2)^2 |\Upsilon(x)|^2}{t(a^2(u)-a_{n+1}^2) L(x)} \right) \\
\end{align*}
Now, integrating over the law of $\zeta(x)$ and using that $\mathbb{E}|\zeta(x)| \le CV(x)$, we obtain
$$ \dtv(\bar{X}^{x,n}_t, \bar{Y}^x_{t,u}) \le C a_n^{-1} (a_n-a_{n+1}) + C\sqrt{t}(1+V^{1/2}(x)). $$
In the case $d=1$, we use \cite[Theorem 1.3]{devroye2018} and obtain the same bounds.

\medskip

$\bullet$ \textbf{Conclusion:} Considering \eqref{eq:dTV_X_Y_bar}, we get
\begin{align*}
\dtv(X^{x,n}_t, \bar{Y}^x_{t,u}) & \le C(1+|x|^2)t + Ce^{Ca_{n+1}^{-1}(1+|x|^2)}t^{1/2} + Ce^{C|x|^2}t^{r/(2r+1)} \\
& \quad + C a_n^{-1} (a_n-a_{n+1}) + C\sqrt{t}(1+V^{1/2}(x)) \\
& \le Ce^{Ca_{n+1}^{-1}(1+|x|^2)}t^{r/(2r+1)} + C a_n^{-1} (a_n-a_{n+1}).
\end{align*}

\end{proof}

\subsection{Total variation bound in small time for the continuous SDE}

\begin{proposition}
\label{prop:dtv_X_Y_st}
Assume that $\sigma \in \mathcal{C}^{2r}_b$ and let $\bar{\gamma} > 0$. There exists $C>0$ such that for all $\varepsilon > 0$, $n \ge 0$, $u$, $t\ge 0$ such that $u \in [T_n,T_{n+1}]$, $u+t \in [T_{n},T_{n+1}]$ and $t \le \bar{\gamma}$,
\begin{equation}
\label{eq:dtv_X_Y_st}
\dtv(X^{x,n}_t,Y^x_{t,u}) \le Ce^{Ca_{n+1}^{-1}(1+|x|^2)}t^{1/2} + Ca_{n+1}^{-(d+r)}(a(u)-a_{n+1})^{2r/(2r+1)} .
\end{equation}
\end{proposition}
%
%We use \cite[Theorem 2.4]{qian2003} extended to non-homogeneous diffusion processes stated below.
%\begin{proposition}
%\label{prop:qian-extended}
%Let $Y$ and $Z$ be the solutions of the generic non-homogeneous SDEs
%\begin{align*}
%& Y^x_0 = x, \quad dY^x_t = b_{_Y}(t,Y^x_t)dt + \sigma_{_Y}(t,Y^x_t)dW_t, \\
%& Z^x_0 = x, \quad dZ^x_t = \sigma_{_Y}(t,Z^x_t)dW_t
%\end{align*}
%where $b_Y$ is bounded, Lipschitz-continuous and $\sigma_Y$ is elliptic, bounded and Lipschitz-continuous uniformly in $t \in [0,T]$.
%Then we have for every $t \in [0,T]$, $x,y \in \mathbb{R}^d$,
%\begin{equation}
%\label{eq:girsanov_p}
%p_{_Y}(0,t,x,y) = p_{_Z}(0,t,x,y) + \int_0^t \mathbb{E}\left[U_s^x \langle b_{_Y}(s,Z_s^x), \nabla_x p_{_Z}(s,t,Z_s^x,y) \rangle \right] ds,
%\end{equation}
%where $p(s,t,x,y)$ denotes the transition probability from $x$ to $y$ between times $s$ and $t$ and where $U$ is defined as
%\begin{align*}
%& U_s^x = \exp\left( \int_0^s \langle g(u,Z_u^x) b_{_Y}(u,Z_u^x), dZ_u^x \rangle - \frac{1}{2} \int_0^s \langle g(u,Z_u^x) b_Y(u,Z_u^x), b_Y(u,Z_u^x) \rangle du \right), \quad g = (\sigma \sigma^\top)^{-1}.
%\end{align*}
%\end{proposition}

%
%We now prove Proposition \ref{prop:dtv_X_Y_st}.
\begin{proof}
We have
\begin{align}
\dtv(X^{x,n}_t, Y^x_{t,u}) & \le \dtv(X^{x,n}_t, \widetilde{X}^{x,n}_t) + \dtv(\widetilde{X}^{x,n}_t, Z^{x,n}_t) + \dtv(Z^{x,n}_t, \widetilde{Z}^x_{t,u}) \nonumber \\
\label{eq:dTV_X_Y_st}
& \quad + \dtv(\widetilde{Z}^x_{t,u}, \widetilde{Y}^x_{t,u}) + \dtv(\widetilde{Y}^x_{t,u}, Y^x_{t,u})
\end{align}
where
\begin{align*}
& d\widetilde{X}^{x,n}_t = \tilde{b}^x_{a_{n+1}}(\widetilde{X}^{x,n}_t)dt + a_{n+1}\sigma(\widetilde{X}^{x,n}_t)dW_t , \\
& dZ^{x,n}_t = a_{n+1}\sigma(Z^{x,n}_t)dW_t, \\
& d\widetilde{Z}^x_{t,u} = a(u+t)\sigma(\widetilde{Z}^x_{t,u})dW_t , \\
& d\widetilde{Y}^x_{t,u} = \tilde{b}^x_{a(u+t)}(\widetilde{Y}^x_{t,u})dt + a(u+t)\sigma(\widetilde{Y}^x_{t,u})dW_t.
\end{align*}
Using \cite[Lemma 3.2]{bras2021-3}, we have
$$ \dtv(X^{x,n}_t, \widetilde{X}^{x,n}_t) + \dtv(\widetilde{Y}^x_{t,u}, Y^x_{t,u}) \le C(1+|x|^2)t .$$
Using \eqref{eq:dtv_X_tilde_Z} again, we have
$$ \dtv(\widetilde{X}^{x,n}_t, Z^{x,n}_t) \le Ce^{C a_{n+1}^{-1}(1+|x|^2)}t^{1/2} .$$
Moreover, using \cite[Theorem 2.4]{qian2003} (with an immediate adaptation to the non-homogeneous case) and establishing the same bounds as in \cite[Lemma 3.5]{bras2021-3}, we also have
\begin{equation}
\label{eq:dtv_Z_tilde_Y_tilde}
\dtv(\widetilde{Z}^x_{t,u}, \widetilde{Y}^x_{t,u}) \le Ce^{C a_{n+1}^{-1}(1+|x|^2)}t^{1/2} .
\end{equation}
We now turn to $\dtv(Z^{x,n}_t, \widetilde{Z}^x_{t,u})$.
%a strategy similar to \cite[Theorem 2.1]{bras2021-3}. That is, we introduce artificial regularization. For $\zeta \sim \mathcal{N}(0,I_d)$ and $\varepsilon>0$, let us write
%\begin{align*}
%\dtv(Z^{x,n}_t, \widetilde{Z}^x_{t,u}) & \le \dtv(Z^{x,n}_t, Z^{x,n}_t + \sqrt{\varepsilon} \zeta) + \dtv(Z^{x,n}_t + \sqrt{\varepsilon} \zeta, \widetilde{Z}^x_{t,u} + \sqrt{\varepsilon} \zeta) + \dtv(\widetilde{Z}^x_{t,u} + \sqrt{\varepsilon} \zeta, \widetilde{Z}^x_{t,u}) \\
%& =: D_2 + D_3 + D_4 .
%\end{align*}
Using Lemma \ref{lemma:change_time_martingale} as in \eqref{eq:M_sigma_density_bound} we have
$$ |\nabla^{2r}_y p_{_Z}(t,x,y)| = |\nabla^{2r}_y p_{_{M(\sigma)}}(a_{n+1}^{2}t,x,y)| \le C\frac{a_{n+1}^{-(d+r)}}{t^{(d+r)/2}} e^{-c|x-y|^2/t} .$$
%and then
%$$ D_2 \le Ca_{n+1}^{-(d+2)/2} \varepsilon t^{-1} .$$
To bound $p_{{\widetilde{Z}}}$ we use the change of time $F$ satisfying $F'(t) = a^{-2}(u+F(t))$ so that
$$ a_n^{-2} t \le F(t) \le a_{n+1}^{-2} t \quad \text{and} \quad a_{n+1}^{2} t \le F^{(-1)}(t) \le a_n^{2} t $$
and then
$$ |\nabla^{2r}_y p_{_{\widetilde{Z}}}(t,x,y)| = |\nabla^{2r}_y p_{_{M(\sigma)}}(F^{(-1)}(t),x,y)| \le C\frac{a_{n+1}^{-(d+r)}}{t^{(d+r)/2}} e^{-c|x-y|^2/t} .$$
%which in turns implies
%$$ D_4 \le Ca_{n+1}^{-(d+2)} \varepsilon t^{-1} .$$
We prove as in \cite[Lemma 6.2]{bras2021-2} that
$$ \| Z^{x,n}_t - \widetilde{Z}^x_{t,u} \|_1 \le C(a(u)-a_{n+1})t^{1/2} $$
and then using \cite[Proposition 2.3]{bras2021-3} we get for every $\varepsilon>0$:
$$ \dtv(Z^{x,n}_t, \widetilde{Z}^x_{t,u}) \le Ca_{n+1}^{-(d+r)}\varepsilon^r t^{-r} + C\varepsilon^{-1/2}(a(u)-a_{n+1})t^{1/2}. $$
Choosing $\varepsilon = (a(u)-a_{n+1})^{2/(2r+1)} t$ yields
$$ \dtv(Z^{x,n}_t, \widetilde{Z}^x_{t,u}) \le Ca_{n+1}^{-(d+r)}(a(u) - a_{n+1})^{2r/(2r+1)} .$$

$\bullet$ \textbf{Conclusion:} considering \eqref{eq:dTV_X_Y_st}, we get
$$ \dtv(X^{x,n}_t,Y^x_{t,u}) \le C(1+|x|^2)t + Ce^{Ca_{n+1}^{-1}(1+|x|^2)}t^{1/2} + C a_{n+1}^{-(d+r)}(a(u)-a_{n+1})^{2r/(2r+1)} .$$
\end{proof}

\begin{remark}
As in \cite[Theorem 2.3]{bras2021-3}, we could improve the dependency in $|x|$ in \eqref{eq:dtv_X_Y_bar_st} and \eqref{eq:dtv_X_Y_st}, at the expanse of further assumptions on $V$. However it would require to track the dependency in the ellipticity (in $a_n$) in the bounds proved in \cite{menozzi2021}, which rely on Malliavin calculus. We believe that it would considerably increase the length and the technicality of the present article, while bringing no significant improvement to our final results.
\end{remark}

\section{Convergence of the plateau SDE $X_t$ in total variation}
\label{sec:X_convergence}

%Having established the bounds in small time, the rest of the proof of Theorem \ref{thm:main} follows the same strategy as for the proof of the convergence in $L^1$-Wasserstein distance in \cite{bras2021-3}.
In this section, we prove the convergence of the plateau SDE $(X_t)$ defined in \eqref{eq:def_X}.

\subsection{Exponential contraction in total variation}
We first show that the property of exponential contraction that holds for the $L^1$-Wasserstein distance under the setting described in Section \ref{subsec:assumptions} (see \cite[Theorem 4.2]{bras2021-2}) also holds for the total variation distance.
\begin{theorem}
\label{thm:contraction_dTV}
Let $X$ be the solution to
\begin{equation}
\label{eq:X_constant_a}
X_0^x = x, \quad dX_t^x = b_a(X_t^x)dt + a\sigma(X_t^x) dW_t,
\end{equation}
with $a \in (0,A]$ and where $b_a$ is defined in \eqref{eq:def_b}, so that $\nu_a$ defined in \eqref{eq:def_nu} is the unique invariant distribution of $X$ (\cite[Proposition 2.5]{pages2020}). Let $t_0 \in (0,1]$.
Under the assumption \eqref{Eq:eq:V_confluence},
\begin{enumerate}[label=(\alph*)]
\item For every $x$, $y \in \mathbb{R}^d$ and for every $t \ge t_0$ we have
\begin{equation}
\label{eq:contraction_dtv}
\dtv(X^x_t, X^y_t) \le Ca^{-1} e^{C_1/a^2} e^{-\rho_a t}|x-y|, \quad \rho_a := e^{-C_2/a^2} .
\end{equation}
\item For every $x \in \mathbb{R}^d$ and for every $t \ge t_0$ we have
\begin{equation}
\label{eq:contraction_dtv_nu}
\dtv(X^x_t,\nu_a) \le C a^{-1} e^{C_1/a^2} e^{-\rho_a t} \nu_a(|x-\cdot|).
\end{equation}
\end{enumerate}
\end{theorem}
\begin{proof}
(a) Following \cite[Theorem 4.2]{bras2021-2}, we have
\[ \forall x,y \in \mathbb{R}^d, \ \mathcal{W}_1\left(X^x_t, X^y_t\right) \le C e^{C_1/a^2} |x-y|e^{-\rho_a t}. \]
Let $t \ge t_0$ and let $f : \mathbb{R}^d \to \mathbb{R}$ a Borel bounded function. Then
$$ \mathbb{E}[f(X^x_t)] - \mathbb{E}[f(X^y_t)] = \mathbb{E}[P_{t_0}^X f(X^x_{t-t_0})] - \mathbb{E}[P_{t_0}^X f(X^x_{t-t_0})],$$
where $P^X$ denotes the kernel associated to $X$.
But using \cite[Proposition 3.1]{pages2020} we have for every $z_1$ and $z_2 \in \mathbb{R}^d$,
$$ P_{t_0}^X f(z_2) - P_{t_0}^X f(z_1) = \langle \nabla P_{t_0}^X f(\xi), z_2 - z_1 \rangle = \frac{1}{t_0} \mathbb{E}\left[f(X_t^\xi) \left\langle \int_0^{t_0}(a^{-1}\sigma^{-1}(X^\xi_s)Y^\xi_s)^\top dW_s, z_2 - z_1 \right\rangle \right] ,$$
where $\xi \in (z_1,z_2)$ and $(Y^\xi_s)_{s \ge 0}$ denotes the tangent process of $(X^\xi_s)$, i.e.
\begin{equation}
Y^\xi_0 = I_d, \quad dY^\xi_s = \nabla b_a(X^\xi_s) Y^\xi_s ds + a \nabla \sigma(X^\xi_s) Y^\xi_s \otimes dW_s.
\end{equation}
Since $\nabla \sigma$ and $\nabla b_a$ are bounded (uniformly in $a$), we have
$$ \textstyle \sup_{\xi \in \mathbb{R}^d, s \in [0,t_0]} \mathbb{E}\| Y_s^\xi \|^2_2 < + \infty, $$
where the bound does not depend on $a$. So that
$$ P_{t_0}^X f(z_2) - P_{t_0}^X f(z_1) \le C ||f||_\infty |z_2-z_2|a^{-1} \sup_{\xi \in \mathbb{R}^d, s \in [0,t_0]} \mathbb{E}\| Y_s^\xi \|_2 ,$$
and then $[P_{t_0}^X f]_{\text{Lip}} \le C a^{-1} \|f\| _\infty$. Then we obtain
$$ \dtv(X^x_t, X^y_t) \le Ca^{-1} \mathcal{W}_1(X^x_{t-t_0},X^y_{t-t_0}) \le C a^{-1} e^{C_1/a^2} e^{-\rho_a t} |x-y| .$$

\medskip

\noindent (b) As $\nu_a$ is the invariant distribution of the diffusion \eqref{eq:X_constant_a} we have
\begin{align*}
\dtv\left(X^x_t,\nu_a\right) & \le \int_{\mathbb{R}^d} \dtv\left(X^x_t,X^y_t\right) \nu_a(dy) \le C e^{C_1/a^2} e^{-\rho_a t} \int_{\mathbb{R}^d} |x-y| \nu_a(dy) \\
& \le C e^{C_1/a^2}e^{-\rho_a t} \nu_a(|x-\cdot|) .
\end{align*}
\end{proof}

\subsection{Convergence of the plateau SDE}

Let $(T_n)$ be the time schedule defined in \eqref{eq:def_T_n} and by a slight abuse of notation we define
\begin{equation}
\label{eq:def_a_n}
a_n := a(T_n) = \frac{A}{\sqrt{\log(T_n+e)}} \quad \text{and} \quad \rho_n := \rho_{a_n} = e^{-C_2/a_n^2}.
\end{equation}
We recall that \cite[Lemma 4.3]{bras2021-2}
\begin{equation}
0 \le a_n - a_{n+1} \asymp (n \log^{3/2}(n))^{-1} .
\end{equation}

\begin{proposition}
\label{prop:dtv_nu_an}
Let $\nu_a$, $a \in (0,A]$, be the Gibbs measure defined in \eqref{eq:def_nu}. Assume that $V$ is coercive, that $(x \mapsto |x|^2 e^{-2V(x)/A^2}) \in L^1(\mathbb{R}^d)$ and \eqref{Eq:eq:min_V}. Then for $n \ge 2$,
\begin{equation}
\dtv(\nu_{a_n},\nu_{a_{n+1}}) \le \frac{C}{n \log(n)} .
\end{equation}
Moreover, for every $s$, $t \in [a_{n+1},a_n]$, we have
\begin{equation}
\dtv(\nu_s,\nu_t) \le \frac{C}{n \log(n)} .
\end{equation}
\end{proposition}

The proof is given in the Appendix \ref{sec:app_proof_dtv}.

We now prove the convergence of the SDE "by plateaux" for the total variation distance.
\begin{theorem}
\label{thm:convergence_X}
Let $X$ be the process defined in \eqref{eq:def_X} and \eqref{eq:def_X:2}. Let $t_0$ be defined as in Theorem \ref{thm:contraction_dTV}.
If $A > \max(\sqrt{(1+\beta^{-1})C_2},$ $\sqrt{(1+\beta)C_1})$ where $C_1$ and $C_2$ are defined in \eqref{eq:contraction_dtv}, then for all $x_0 \in \mathbb{R}^d$ and for all $C_{(T)}' < C_{(T)}$, for all large enough $n \ge n(C_{(T)}')$, on the time schedule $(T_n)$ we have
\begin{equation}
\dtv(X^{x_0}_{T_{n}},\nu_{a_{n}}) \le C a_n^{-1} n^{-1+(\beta+1)C_1/A^2} \exp\left(-(C_{(T)}')^{1-C_2/A^2}(\beta+1)n^{\beta-(\beta+1)C_2/A^2}\right) (1+|x_0|)
\end{equation}
and for every $t \in \mathbb{R}^+ \setminus (\bigcup_{n \ge 1} [T_n, T_n+t_0])$ we have
\begin{equation}
\dtv(X^{x_0}_t,\nu_{a(t)}) \le \frac{C(1+|x_0|)}{t^{(1+\beta)^{-1}-C_1/A^2} \log(t+e)}.
\end{equation}
\end{theorem}
\begin{proof}
For fixed $x \in \mathbb{R}^d$ and using Theorem \ref{thm:contraction_dTV}, we have for every bounded Borel function $f : \mathbb{R}^d \to \mathbb{R}$,
$$ \mathbb{E}[f(X^{x,n}_{T_{n+1}-T_n})] - \mathbb{E}[f(Z_{a_{n+1}})] \le Ca_{n+1}^{-1}e^{C_1/a_{n+1}^2}e^{-\rho_{a_{n+1}}(T_{n+1}-T_n)} \|f\|_\infty \mathbb{E}|x-Z_{a_{n+1}}| ,$$
where $Z_{a_{n+1}} \sim \nu_{a_{n+1}}$. Now integrating $x$ with respect to the law of $X^{x_0}_{T_n}$ yields
\begin{align}
\dtv(X^{x_0}_{T_{n+1}},\nu_{a_{n+1}}) & \le C a_{n+1}^{-1}e^{C_1/a_{n+1}^2}e^{-\rho_{a_{n+1}}(T_{n+1}-T_n)} \left( \mathcal{W}_1(X^{x_0}_{T_n},\nu_{a_{n}}) + \mathcal{W}_1(\nu_{a_{n+1}},\nu_{a_n})\right) \nonumber \\
& \le C \frac{a_{n+1}^{-1}\mu_{n+1}}{n\log^{3/2}(n)}(1+|x_0|), \nonumber \\
\label{eq:def_mu}
& \mu_n := e^{C_1/a_{n+1}^2}e^{-\rho_{a_{n+1}}(T_{n+1}-T_n)}
\end{align}
where we used \cite[Theorem 5.1]{bras2021-2} and \cite[Proposition 4.4]{bras2021-2}. We use the bound on $\mu_n$ given by \cite[(5.5)]{bras2021-2}. Then to bound $\dtv(X^{x_0}_t, \nu_{a_{n+1}})$ for any $t \in (T_n+t_0, T_{n+1})$, we apply Theorem \eqref{thm:contraction_dTV} on the time interval $[T_n, t]$ which length is not smaller than $t_0$ and we conclude as in the proof of \cite[Theorem 5.1]{bras2021-2}.
\end{proof}

\begin{remark}
The condition that $t$ does not belong in any interval $[T_n,T_n+t_0]$ is a technical condition which is specific to our strategy of proof. However this condition is not a problem for the convergence of $Y_t$ and $\bar{Y}_t$ since for these two processes, the time schedule $(T_n)$ is only a tool for the proof.
\end{remark}

\section{Convergence of $Y_t$ in total variation}
\label{sec:Y_convergence}

We now consider $(Y_t)$ as defined in \eqref{eq:def_Y} with extended definition \eqref{eq:def_Y:2}.

\subsection{Preliminary lemmas}

\begin{lemma}
\label{lemma:D.1a:exp}
Let $\lambda \in \mathbb{R}^+$. There exists $C>0$ such that for every $n \ge 0$, $u \ge 0$ and every $x \in \mathbb{R}^d$:
\begin{equation}
\sup_{t \ge 0} \mathbb{E} \left[e^{\lambda |X^{x,n}_t|^2}\right] \le Ce^{\lambda |x|^2} \ \text{ and } \ \sup_{t \ge 0} \mathbb{E} \left[e^{\lambda |Y^x_{t,u}|^2}\right] \le Ce^{\lambda|x|^2} .
\end{equation}
\end{lemma}
\begin{proof}[Sketch of proof]
By It\=o's Lemma, we have for $k \ge n$ and for $t \in [T_k-T_n,T_{k+1}-T_n)$:
\begin{align*}
d\left(e^{\lambda|X_t^{x,n}|^2}\right) & = \lambda e^{\lambda|X_t^{x,n}|^2} \left(2\langle X_t^{x,n}, dX_t^{x,n} \rangle + d\langle X^{x,n} \rangle_t \right) + 2\lambda^2 |X_t^{x,n}|^2 e^{\lambda|X_t^{x,n}|^2} d\langle X^{x,n}\rangle_t \\
& = \lambda e^{\lambda|X_t^{x,n}|^2} \Big( -2\langle \sigma \sigma^\top \nabla V(X_t^{x,n}), X_t^{x,n} \rangle dt + 2a_{n+1}^2 \langle X_t^{x,n}, \Upsilon(X_t^{x,n}) \rangle dt \\
& \quad + 2a_{n+1} \langle X_t^{x,n}, \sigma(X_t^{x,n}) dW_t \rangle + a_{n+1}^2 \tr(\sigma\sigma^\top(X_t))dt \Big) \\
& \quad	+ 2\lambda^2 e^{\lambda|X_t^{x,n}|^2} a_{n+1}^2 (X_t^{x,n})^\top \sigma \sigma^\top(X_t^{x,n}) X_t^{x,n} dt
\end{align*}
the "dominating" term is $-\langle \sigma \sigma^\top \nabla V(X_t^{x,n}), X_t^{x,n} \rangle dt$ which makes $\mathbb{E} [e^{\lambda |X^{x,n}_t|^2}]$ decrease. Using assumption \eqref{Eq:eq:V_confluence}, we have for $|X_t^{x,n}|$ large enough,
$$ -\langle \sigma \sigma^\top \nabla V(X_t^{x,n}), X_t^{x,n} \rangle \le -C\ubar{\sigma}_0^2 \alpha_0 |X_t^{x,n}|^2. $$
Moreover, using the facts that $\Upsilon$ and $\sigma$ are bounded, that $a_n \to 0$, that $|\nabla V| \le CV^{1/2}$ and that $\sigma \sigma^\top \ge \ubar{\sigma}_0^2 I_d$, for large enough $|X^{x,n}_t|$ and large enough $n$, the coefficient in $dt$ in the last equation is negative. We deal with the cases where $|X^{x,n}_t|$ is not large enough or where $n$ is not large enough the same way as in the proof of \cite[Lemma 6.1]{bras2021-2} and \cite[Lemma 7.1]{bras2021-2}, where more details can be found.

The proof is the same for $Y$, replacing $a_{k+1}$ by $a(u+t)$.
\end{proof}

\begin{proposition}
\label{prop:3.6:Y_dTV}
Let $T$, $\bar{\gamma}>0$. There exists $C>0$ such that for every Borel bounded function $f : \mathbb{R}^d \to \mathbb{R}$ and every $t \in (0,T]$, for all $n \ge 0$, for all $\gamma < \bar{\gamma}$ such that $u \in [T_n,T_{n+1}]$ and $u+t+\gamma \in [T_n,T_{n+1}]$,
\begin{equation}
\left| \mathbb{E}\left[P_t^{X,n} f(Y_{\gamma,u}^x)\right]  - \mathbb{E}\left[P_t^{X,n} f(X_\gamma^{x,n})\right] \right| \le C a_{n+1}^{-2}(a_n-a_{n+1}) \|f\|_\infty \gamma t^{-1} V(x).
\end{equation}
\end{proposition}
\begin{proof}
We apply \cite[Proposition 6.4]{bras2021-2} to $g_t := P_t^{X,n} f$ with $t>0$. Following \cite[Proposition 3.2(b)]{pages2020}, we have
$$ \Phi_{g_t}(x) \le C \|f\|_\infty a_{n+1}^{-2} t^{-1} \max\left(V^{1/2}(x), \left|\left|\sup_{\xi \in (X^{x,n}_\gamma, Y_{\gamma,u}^x)} V^{1/2}(\xi) \right|\right|_2, V^{1/2}(x)\left|\left|\sup_{\xi \in (x, X_\gamma^{x,n})} V^{1/2}(\xi) \right|\right|_2 \right). $$
We conclude as in the proof of \cite[Proposition 6.5]{bras2021-2}.

%But following \eqref{Eq:eq:V_assumptions}, $\nabla V/V^{1/2}$ is bounded so $V^{1/2}$ is Lipschitz-continuous and then
%\begin{align*}
%& \left|\left|\sup_{\xi \in (x, X_\gamma^{x,n})} V^{1/2}(\xi) \right|\right|_2 \le \left|\left| V^{1/2}(x) + [V^{1/2}]_{\text{Lip}} |X_\gamma^{x,n}-x| \right|\right|_2 \le CV^{1/2}(x) \\
%& \left|\left|\sup_{\xi \in (X^{x,n}_\gamma, Y_{\gamma,u}^x)} V^{1/2}(\xi) \right|\right|_2 \le \left|\left| V^{1/2}(x) + [V^{1/2}]_{\text{Lip}} \max(|X^{x,n}_\gamma-x|, |Y^x_{\gamma,u}-x|) \right|\right|_2 \le CV^{1/2}(x),
%\end{align*}
%where we used \cite[Lemmas 6.2 and 6.3]{bras2021-2}. We thus obtain the desired result.
\end{proof}

%We now follow the lines of Section \ref{subsec:proof_Y}. We use Corollary \ref{cor:gibbs} for the total variation, which yields a supplementary factor $a_{n+1}^{-1}$ and we use Proposition \ref{prop:3.6:Y_dTV} instead of Proposition \ref{prop:3.6:Y}, yielding $t^{-1}$ instead of $t^{-1/2}$.

\subsection{Proof of Theorem \ref{thm:main}(a)}
\label{subsec:proof_Y}

More precisely, we prove that for all $\beta >0$, if
\begin{equation}
\label{eq:A_condition_Y}
A > \max\left(\sqrt{(\beta+1)(2C_1+C_2)}, \sqrt{(1+\beta^{-1})C_2} \right) ,
\end{equation}
then
\begin{equation}
\dtv\left(Y^{x_0}_t, \nu_{a(t)}\right) \le \frac{Ce^{C\sqrt{\log(t)}(1+|x_0|^2)}}{t^{(1+\beta)^{-1}-(2C_1+C_2)/A^2} }.
\end{equation}
\begin{proof}
We follow the proof of \cite[Theorem 2.1(b)]{bras2021-2} in \cite[Section 7.3]{bras2021-2} based on a domino strategy with respect to some decreasing step sequence $(\gamma_n)$, even though $Y$ is not an Euler-Maruyama scheme. In this case, the step sequence $(\gamma_n)$ is only a tool for the proof. This way we can choose freely the sequence $(\gamma_n)$ in this section.
We use Theorem \ref{thm:contraction_dTV} in place of \cite[Theorem 4.2]{bras2021-2} and Proposition \ref{prop:3.6:Y_dTV} in place of \cite[Proposition 7.4]{bras2021-2}.
For $f:\mathbb{R}^d \to \mathbb{R}$ bounded measurable
%we split $|\mathbb{E}f(X^{x,n}_{T_{n+1}-T_n}) - \mathbb{E}f(Y^x_{T_{n+1}-T_n, T_n})|$ into four terms $(c^{\text{init}})$, $(a)$, $(b)$, $(c^{\text{end}})$. We have
and for $x \in \mathbb{R}^d$ we write
\begin{align*}
& \left| \mathbb{E}f(X_{T_{n+1}-T_n}^{x,n}) - \mathbb{E}f(Y_{T_{n+1}-T_n,T_n}^{x})\right| \le \left|(P^{Y}_{\gamma^{\text{init}},T_n} - P^{X,n}_{\gamma^{\text{init}}}) \circ P^{X,n}_{T_{n+1}-\Gamma_{N(T_n)+1}} f(x)\right| \\
& + \sum_{k=N(T_n)+2}^{N(T_{n+1}-T)} \left| P^{Y}_{\gamma^{\text{init}},T_n} \circ P^{Y}_{\gamma_{N(T_n)+2},\Gamma_{N(T_n)+1}} \circ \cdots \circ P^{Y}_{\gamma_{k-1},\Gamma_{k-2}} \circ (P^{Y}_{\gamma_{k},\Gamma_{k-1}} - P^{X,n}_{\gamma_k}) \circ P^{X,n}_{T_{n+1} - \Gamma_k} f(x) \right| \\
& + \sum_{k=N(T_{n+1}-T)+1}^{N(T_{n+1})-1} \left| P^{Y}_{\gamma^{\text{init}},T_n} \circ P^{Y}_{\gamma_{N(T_n)+2},\Gamma_{N(T_n)+1}} \circ \cdots \circ P^{Y}_{\gamma_{k-1},\Gamma_{k-2}} \circ (P^{Y}_{\gamma_{k},\Gamma_{k-1}} - P^{X,n}_{\gamma_k}) \circ P^{X,n}_{T_{n+1} - \Gamma_k} f(x) \right| \\
& {+} \left| P^{Y}_{\gamma^{\text{init}},T_n} {\circ} P^{Y}_{\gamma_{N(T_n)+2},\Gamma_{N(T_n)+1}} {\circ} \cdots {\circ} P^{Y}_{\gamma_{N(T_{n+1})-1},\Gamma_{N(T_{n+1})-2}} {\circ} (P^{Y}_{\gamma^{\text{end}}+\gamma_{N(T_{n+1})},\Gamma_{N(T_{n+1})-1}} {-} P^{X,n}_{\gamma^{\text{end}}+\gamma_{N(T_{n+1})}}) f(x)\right| \\
& =: (c^{\text{init}}) + (a) + (b) + (c^{\text{end}}),
\end{align*}
where
$$\gamma^{\text{init}} := \Gamma_{N(T_n)+1}-T_n \le \gamma_{N(T_n)+1} \quad \text{ and } \quad \gamma^{\text{end}} := T_{n+1}-\Gamma_{N(T_{n+1})} \le \gamma_{N(T_{n+1})+1}. $$
Then we have
\begin{align*}
(a) & \le C a_{n+1}^{-3} e^{C_1 a_{n+1}^{-2}} e^{-\rho_{n+1} T_{n+1}} \|f\|_\infty V(x) (a_n-a_{n+1}) \sum_{k=N(T_n)+2}^{N(T_{n+1}-T)} \gamma_k e^{\rho_{n+1} \Gamma_k}  \\
& \le C a_{n+1}^{-3} e^{C_1 a_{n+1}^{-2}} \|f\|_\infty (a_n - a_{n+1}) V(x) \rho_{n+1}^{-1} .
\end{align*}
We obtain likewise
$$ (c^{\text{init}}) \le C a_{n+1}^{-3} e^{-\rho_{n+1}(T_{n+1}-T_n)} \|f\|_\infty (a_n - a_{n+1}) \gamma_{N(T_n)} V(x) .$$
Applying Proposition \ref{prop:3.6:Y_dTV} yields
\begin{align*}
(b) & \le C a_{n+1}^{-2} (a_n - a_{n+1}) \|f\|_\infty V(x) \sum_{k=N(T_{n+1}-T)+1}^{N(T_{n+1})-1} \frac{\gamma_k}{T_{n+1} - \Gamma_k} \\
& \le C a_{n+1}^{-2} (a_n - a_{n+1}) \|f\|_\infty V(x) \log(1/\gamma_{N(T_{n+1})}).
\end{align*}
Applying Proposition \ref{prop:dtv_X_Y_st} with $r=1$ along with Lemma \ref{lemma:D.1a:exp} yields
$$ (c^{\text{end}}) \le C \|f\|_\infty \left( e^{Ca_{n+1}^{-1}(1+|x|^2)} \gamma_{N(T_n)}^{1/2} + a_{n+1}^{-(d+1)}(a(T_{n+1}-\gamma_{N(T_{n+1})}) -a_{n+1})^{2/3}\right). $$
But we have
$$ a(T_{n+1}{-}\gamma_{N(T_{n+1})}) -a_{n+1} = a(T_{n+1}{-}\gamma_{N(T_{n+1})}) -a(T_{n+1}) \le C\frac{da}{dt}(T_{n+1}) \cdot \gamma_{N(T_{n+1})} \le \frac{C\gamma_{N(T_{n+1})}}{T_{n+1}} .$$
We now choose $\gamma_n = \gamma_1 n^{-2/3}$ so that $\gamma_{N(T_n)} \asymp n^{-2}$ and then
$$ (c^{\text{end}}) \le C e^{Ca_{n+1}^{-1}(1+|x|^2)} n^{-1} . $$
This way we obtain for every $x \in \mathbb{R}^d$:
\begin{align}
|\mathbb{E}f(X^{x,n}_{T_{n+1}-T_n}) - \mathbb{E}f(Y^x_{T_{n+1}-T_n, T_n})| & \le C\|f\|_\infty \underbrace{ a_{n+1}^{-3} e^{C_1 a_{n+1}^{-2}} (a_n - a_{n+1}) V(x) \rho_{n+1}^{-1} }_{=:v_{n+1}} e^{Ca_{n+1}^{-1}(1+|x|^2)}.
\end{align}
We integrate this inequality with respect to the laws of $X^{x_0}_{T_n}$ and $\bar{Y}_{T_n}^{x_0}$ and obtain, temporarily setting $x_n := X^{x_0}_{T_n}$ and $y_n := Y_{T_n}^{x_0}$ and using \cite[Lemma 6.1]{bras2021-2} and Lemma \ref{lemma:D.1a:exp},
\begin{align*}
\dtv(X^{x_0}_{T_{n+1}}, Y^{x_0}_{T_{n+1}}) & \le \dtv(X_{T_{n+1}-T_n}^{x_n,n}, X_{T_{n+1}-T_n}^{y_n,n}) + \dtv(X_{T_{n+1}-T_n}^{y_n,n}, Y_{T_{n+1}-T_n,T_n}^{\bar{y}_n}) \\
& \le \underbrace{Ca_{n+1}^{-1} e^{C_1 a_{n+1}^{-2}} e^{-\rho_{n+1} (T_{n+1}-T_n)}}_{:= \mu_{n+1}' = a_{n+1}^{-1} \mu_{n+1}} \dtv(X^{x_0}_{T_n}, Y^{x_0}_{T_n}) + \underbrace{Cv_{n+1} e^{Ca_{n+1}^{-1}(1+|x_0|^2)}}_{:=w_{n+1}},
\end{align*}
where $\mu_n$ is defined in \eqref{eq:def_mu}. Iterating this inequality yields
\begin{align*}
\dtv(X^{x_0}_{T_{n+1}}, Y^{x_0}_{T_{n+1}}) & \le C(w_{n+1} + \mu_{n+1}' w_n + \cdots + \mu_{n+1}' \cdots \mu_2' w_1) \le Cw_{n+1},
\end{align*}
where we used, since $A$ satisfies \eqref{eq:A_condition_Y}, that $\mu_n' = O(e^{-Cn^\eta})$ for some $\eta>0$ (see \cite[(5.5)]{bras2021-2}) and that $w_n$ is bounded as it converges to $0$. Moreover using Theorem \ref{thm:convergence_X} we have
\begin{equation}
\dtv(Y^{x_0}_{T_{n}}, \nu_{a_{n}}) \le \dtv(X^{x_0}_{T_{n}}, Y^{x_0}_{T_{n}}) + \dtv(X^{x_0}_{T_n}, \nu_{a_n}) \le \frac{Ce^{C\sqrt{\log(n)}(1+|x_0|^2)}}{n^{ 1-(\beta+1)(C_1+C_2)/A^2}} .
\end{equation}
Finally, let us bound $\dtv(X^{x_0}_t, Y^{x_0}_t)$ for any $t \in [T_n, T_{n+1}]$. If $t \in [T_n+t_0, T_{n+1}]$ then we can apply Theorem \ref{thm:contraction_dTV} and we proceed as in the end of \cite[Section 6.3]{bras2021-2}. If $t \in [T_n,T_n+t_0]$, then we consider another shifted time schedule $\bar{T}_n := C_{(T)}n^{1+\beta} + 2t_0$ such that
$$ \bigcup_{i=0}^\infty [T_n,T_n+t_0] \ \cap \ \bigcup_{i=0}^\infty [\bar{T}_n, \bar{T}_n+t_0] = \varnothing .$$
Making use of the new time schedule we obtain as before a bound on $\dtv(Y^{x_0}_t, \nu_{a(t)})$ for every $t \notin \bigcup_{i=0}^\infty [\bar{T}_n, \bar{T}_n+t_0]$. Since the time schedules $(T_n)$ and $(\bar{T}_n)$ are only tools for the proof of convergence of $Y_t$, we then obtain a bound on $\dtv(Y_t, \nu_{a(t)})$ for every $t \in \mathbb{R}^+$.
\end{proof}

\section{Convergence of the Euler-Maruyama scheme in total variation}
\label{sec:Y_bar_convergence}

We now consider $(\bar{Y}_n)$ as in \eqref{eq:def_Y_bar} with extended definition \eqref{eq:def_Y_bar:2}.

\subsection{Preliminary lemmas}

\begin{lemma}
\label{lemma:D.1a_bar:exp}
Let $\lambda \in \mathbb{R}^+$. There exists a constant $C >0$ such that for every $k \ge 0$, for every $u \in [\Gamma_k, \Gamma_{k+1})$ and for every $x \in \mathbb{R}^d$:
\begin{equation}
\sup_{n \ge k+1} \mathbb{E} \left[e^{\lambda |Y^x_{\Gamma_n-u,u}|^2}\right] \le Ce^{\lambda|x|^2} .
\end{equation}
\end{lemma}
\begin{proof}
The prove is the same as for Lemma \ref{lemma:D.1a:exp}. For the adaptation to discrete time, we refer to the proof of \cite[Lemma 7.1]{bras2021-2}.
\end{proof}

\begin{proposition}
\label{prop:3.6:Y_bar_dTV}
Let $T >0$. There exists $C > 0$ such that for every Lipschitz continuous function $f$ and every $t \in (0,T]$, for all $n \ge 0$, for all $\gamma$ such that $\Gamma_k \in [T_n,T_{n+1}]$, $\gamma \le \gamma_{k+1}$ and $\Gamma_k+t+\gamma \in [T_n,T_{n+1}]$,
\begin{align}
& \left| \mathbb{E}\left[P_t f(\bar{Y}_{\gamma,\Gamma_k}^x)\right]  - \mathbb{E}\left[P_t f(X_{\gamma}^{x,n})\right] \right| \nonumber \\
& \quad \le C\|f\|_\infty V^{2}(x) \left(a_{n+1}^{-2}t^{-1}\left(\gamma^2 {+} (a(\Gamma_k){-}a_{n+1}) \gamma\right) + a_{n+1}^{-3}t^{-3/2}\left(\gamma^2 {+} \gamma^{3/2}(a(\Gamma_k)-a_{n+1})\right) \right).
\end{align}
\end{proposition}
\begin{proof}
The proof is the same as the proof of Proposition \ref{prop:3.6:Y_dTV}, using \cite[Proposition 7.3]{bras2021-2}. We also remark that we can directly improve the bound in $(a_n-a_{n+1})$ into $(a(\Gamma_k)-a_{n+1})$.
\end{proof}

\subsection{Proof of Theorem \ref{thm:main}(b)}

More precisely, we prove that for all $\beta > 0$, if $\sigma \in \mathcal{C}^{2r}_b$ and if
\begin{equation}
\label{eq:A_condition_Y_bar}
A > \max\left(\sqrt{(\beta+1)(2C_1+C_2)}, \sqrt{(1+\beta^{-1})C_2} \right)
\end{equation}
and if $A$ is large enough so that
\begin{equation}
n^{(\beta+1)C_1/A^2} \gamma_{N(T_n)}^{r/(2r+1)} \underset{n \to \infty}{\longrightarrow} 0,
\end{equation}
then
\begin{equation}
\dtv(\bar{Y}^{x_0}_t, \nu_{a(t)}) \le C \left( \frac{\log^{1/2}(t)\max\left[V^2(x_0),1+|x_0|\right]}{t^{(\beta+1)^{-1}-(2C_1+C_2)/A^2} } + e^{C\sqrt{\log(t)} (1+|x_0|^2)} t^{C_1/A^2} \gamma_{Ct}^{r/(2r+1)} \right).
\end{equation}

%We now follow the lines of Section \ref{subsec:proof_Y_bar}. We use Corollary \ref{cor:gibbs} for the total variation, which yields a supplementary factor $a_{n+1}^{-1}$ and we use Proposition \ref{prop:3.6:Y_bar_dTV} instead of Proposition \ref{prop:3.6:Y:bar}, yielding $t^{-3/2}$ instead of $t^{-1}$. For the term $(b)$, we write ???
\begin{proof}
We still follow the proof of \cite[Theorem 2.1(b)]{bras2021-2} in \cite[Section 7.3]{bras2021-2} based on a domino strategy, using Theorem \ref{thm:contraction_dTV} in place of \cite[Theorem 4.2]{bras2021-2} and Proposition \ref{prop:3.6:Y_bar_dTV} in place of \cite[Proposition 7.4]{bras2021-2}. Let $n \ge 0$, for $f:\mathbb{R}^d \to \mathbb{R}$ bounded measurable, we split $|\mathbb{E}f(X^{x,n}_{T_{n+1}-T_n}) - \mathbb{E}f(\bar{Y}^x_{T_{n+1}-T_n, T_n})|$ into four terms $(c^{\text{init}})$, $(a)$, $(b)$, $(c^{\text{end}})$.

Using Theorem \ref{thm:contraction_dTV}, \cite[Lemma 7.1]{bras2021-2} and Proposition \ref{prop:3.6:Y_bar_dTV} we get as in \cite[Section 7.3]{bras2021-2}:
\begin{align*}
& (a) \le Ca_{n+1}^{-4} e^{C_1 a_{n+1}^{-2}} \|f\|_\infty (a_n-a_{n+1}) V^2(x) \rho_{n+1}^{-1} . \\
& (c^{\text{init}}) \le  Ca_{n+1}^{-4} e^{C_1 a_{n+1}^{-2}} e^{-\rho_n(T_{n+1}-T_n)} \|f\|_\infty (a_n-a_{n+1}) \gamma_{N(T_n)+1} V^2(x).
\end{align*}

Using Proposition \ref{prop:3.6:Y_bar_dTV} and \cite[Lemma 7.1]{bras2021-2}, we obtain
\begin{align*}
& (b) \le C a_{n+1}^{-3} \left(\gamma_{N(T_{n+1}-T)} {+} \sqrt{\gamma_{N(T_{n+1}-T)}}(a_n {-} a_{n+1})\right) \|f\|_\infty V^2(x) \sum_{k=N(T_{n+1}-T)+1}^{N(T_{n+1})-1} \frac{\gamma_k}{(T_{n+1}{-}\Gamma_k)^{3/2}} \\
& \quad + C a_{n+1}^{-2} \left( \sum_{k=N(T_{n+1}-T)+1}^{N(T_{n+1})-1} \frac{\gamma_{N(T_{n+1}-T)} \gamma_k}{T_{n+1}-\Gamma_k} + \sum_{k=N(T_{n+1}-T)+1}^{N(T_{n+1})-1} \frac{\gamma_k (a(\Gamma_k){-}a_{n+1})}{T_{n+1}-\Gamma_k} \right) \|f\|_\infty V^2(x).
\end{align*}
But we remark that
$$ a(\Gamma_k) - a_{n+1} = a(\Gamma_k) - a(T_{n+1}) \le C\frac{da}{dt}(T_{n+1}) \cdot (\Gamma_k - T_{n+1}) \le \frac{C(\Gamma_k-T_{n+1})}{T_{n+1} \log^{3/2}(T_{n+1})} $$
and then
\begin{align*}
(b) & \le C a_{n+1}^{-3} \left(\gamma_{N(T_{n+1}-T)} + \sqrt{\gamma_{N(T_{n+1}-T)}}(a_n {-} a_{n+1})\right) \|f\|_\infty V^2(x) \int_{T_{n+1}-T}^{T_{n+1}-\gamma_{N(T_{n+1})}} \frac{du}{(T_{n+1} {-} u)^{3/2}} \\
& \quad + C a_{n+1}^{-2} \left(\gamma_{N(T_{n+1}-T)}\int_{T_{n+1}-T}^{T_{n+1}-\gamma_{N(T_{n+1})}} \frac{du}{T_{n+1} {-} u} + \frac{1}{T_{n+1}} \int_{T_{n+1}-T}^{T_{n+1}-\gamma_{N(T_{n+1})}} du \right) \|f\|_\infty V^2(x) \\
& \le C a_{n+1}^{-3} \left(\gamma_{N(T_{n+1}-T)} + \sqrt{\gamma_{N(T_{n+1}-T)}}(a_n-a_{n+1})\right) \|f\|_\infty V^2(x) \gamma_{N(T_{n+1})}^{-1/2} \\
& \quad + C a_{n+1}^{-2} \left(\gamma_{N(T_{n+1})} |\log(\gamma_{N(T_{n+1})})| + T_{n+1}^{-1}\right) \|f\|_\infty V^2(x) \\
& \le Ca_{n+1}^{-3} \left(\gamma_{N(T_{n+1})}^{1/2} + (a_n-a_{n+1}) \right) \|f\|_\infty V^2(x).
\end{align*}
Applying Proposition \ref{prop:dtv_X_Y_bar_st} along with Lemma \ref{lemma:D.1a_bar:exp} yields
\begin{align*}
(c^{\text{end}}) \le C \|f\|_\infty \left( e^{Ca_{n+1}^{-1} (1+|x|^2)} \gamma_{N(T_{n+1})}^{r/(2r+1)} + a_n^{-2}(a_n-a_{n+1}) \right).
\end{align*}
We finally obtain for every $x \in \mathbb{R}^d$:
\begin{align*}
|\mathbb{E}f(X^{x,n}_{T_{n+1}-T_n}) {-} \mathbb{E}f(\bar{Y}^x_{T_{n{+}1}-T_n, T_n})| & \le C \|f\|_\infty \big(  a_{n{+}1}^{-4} e^{C_1 a_{n{+}1}^{-2}} (a_n{-}a_{n{+}1}) V^2(x) \rho_{n{+}1}^{-1} {+} e^{Ca_{n{+}1}^{-1} (1{+}|x|^2)} \gamma_{N(T_{n{+}1})}^{r/(2r{+}1)} \big).
\end{align*}
The same way as in Section \ref{subsec:proof_Y} we get
\begin{align*}
\dtv(\bar{Y}^{x_0}_{T_{n+1}}, \nu_{a_{n+1}}) \le C \left(  a_{n+1}^{-4} e^{C_1 a_{n+1}^{-2}} (a_n{-}a_{n+1}) \max\left[V^2(x_0),1+|x_0|\right] \rho_{n+1}^{-1} + e^{Ca_{n+1}^{-1} (1+|x_0|^2)} \gamma_{N(T_{n+1})}^{r/(2r+1)} \right)
\end{align*}
and, for $t \in [T_n,T_{n+1}]$,
\begin{align*}
\dtv(\bar{Y}^{x_0}_{t}, \nu_{a(t)}) \le C e^{C_1 a_{n+1}^{-2}} \left( a_{n+1}^{-4} e^{C_1 a_{n+1}^{-2}} (a_n{-}a_{n+1}) \max\left[V^2(x_0),1+|x_0|\right] \rho_{n+1}^{-1} {+} e^{Ca_{n+1}^{-1} (1+|x_0|^2)} \gamma_{N(T_{n+1})}^{r/(2r+1)} \right).
\end{align*}

\end{proof}

\appendix

\section{Appendix}

%\subsection{Auxiliary lemmas}

%\begin{lemma}[\cite{friedman}, Chapter 9, Lemma 7]
%\label{lemma:friedman_Lemma_7}
%For $a>0$, $0 < u < t \le T$, $x \in \mathbb{R}^d$, $\xi \in \mathbb{R}^d$, let
%\begin{align*}
%& I_a := \int_{\mathbb{R}^d} \frac{1}{(u(t-u))^{d/2}} \exp\left(-a \left(\frac{|x-y|^2}{t-u} + \frac{|y-\xi|^2}{u} \right) \right) dy.
%\end{align*}
%Then there exists a constant $C>0$ depending only on $d$ and $T$ such that for every $ 0 < \varepsilon < 1$,
%$$ I_a \le \frac{C}{(\varepsilon a t)^{d/2}} \exp \left(-a(1-\varepsilon) \frac{|x-\xi|^2}{t} \right) .$$
%\end{lemma}

\subsection{Proof of Proposition \ref{prop:dtv_nu_an}}
\label{sec:app_proof_dtv}

\begin{proof}
We use the characterization of the total variation distance as the $L^1$-distance between the densities, which reads
\begin{align*}
& \dtv(\nu_{a_n},\nu_{a_{n+1}}) = \int_{\mathbb{R}^d} \left| \mathcal{Z}_{a_n} e^{-2(V(x)-V^\star)/a_n^2} - \mathcal{Z}_{a_{n+1}} e^{-2(V(x)-V^\star)/a_{n+1}^2} \right| dx \\
& \quad \le \mathcal{Z}_{a_{n+1}} \int_{\mathbb{R}^d} \left|e^{-2(V(x)-V^\star)/a_n^2} - e^{-2(V(x)-V^\star)/a_{n+1}^2} \right| dx + |\mathcal{Z}_{a_n} - \mathcal{Z}_{a_{n+1}}| \int_{\mathbb{R}^d} e^{-2(V(x)-V^\star)/a_n^2} dx \\
& \quad = \mathcal{Z}_{a_{n+1}}a_{n+1}^d \int_{\mathbb{R}^d} \left|e^{-2(V(a_{n+1}x)-V^\star)/a_n^2} - e^{-2(V(a_{n+1}x)-V^\star)/a_{n+1}^2} \right| dx \\
& \quad \quad + \left|1 - \frac{\mathcal{Z}_{a_n}}{\mathcal{Z}_{a_{n+1}}}\right| \mathcal{Z}_{a_{n+1}} a_n^d \int_{\mathbb{R}^d} e^{-2(V(a_n x)-V^\star)/a_n^2} dx.
\end{align*}
Using \cite[(B.3)]{bras2021-2} and \cite[(B.5)]{bras2021-2}, the first term is bounded by
\begin{align*}
C\frac{a_n-a_{n+1}}{a_n} \int_{\mathbb{R}^d} e^{-2(V(a_{n+1}y)-V^\star)/a_{n}^2} \frac{V(a_{n+1}y)-V^\star}{a_n^2} dx \le C\frac{a_n-a_{n+1}}{a_n},
\end{align*}
because the integral converges by dominated convergence as for the proof of \cite[(B.3)]{bras2021-2}.
Using \cite[(B.3)]{bras2021-2} and \cite[(B.4)]{bras2021-2}, the second term is bounded by $C(n\log(n))^{-1}$.
\end{proof}

%\bibliographystyle{alpha}
%\bibliography{../langevin-multiplicative-noise/biblio}

\newcommand{\etalchar}[1]{$^{#1}$}

\end{document}